\documentclass[12pt,letter]{amsart}
\usepackage{xcolor}
\usepackage[all,color]{xy}
\usepackage{amsmath,amssymb,amsthm,amscd,amsxtra,amsfonts,mathrsfs,graphicx,enumerate,bm,slashed,array,setspace,amsfonts}
\input xy
\xyoption{all}
\newtheorem{Theorem}{Theorem}[section]
\newtheorem{Lemma}[Theorem]{Lemma}
\newtheorem{Proposition}[Theorem]{Proposition}
\newtheorem{Corollary}[Theorem]{Corollary}
\newtheorem{Example}[Theorem]{Example}
\newtheorem{Remark}[Theorem]{Remark}

\newtheorem{Definition}[Theorem]{Definition}

\newtheorem{Notation}[Theorem]{Notation}
\newtheorem*{Theorem A}{Theorem A}
\newtheorem*{Theorem B}{Theorem B}
\newtheorem*{Theorem C}{Theorem C}

\newcommand*{\overbar}[1]{\mkern 1.5mu\overline{\mkern-1.5mu#1\mkern-1.5mu}\mkern 1.5mu}
\advance\evensidemargin-.5in
\advance\oddsidemargin-.5in
\advance\textwidth1in

\begin{document}
\author{Karin Baur}
\address{School of Mathematics, University of Leeds, Leeds, LS2 9JT, United Kingdom}
\address{On leave from the University of Graz}
\email{pmtkb@leeds.ac.uk}
\author{Charlie Beil}
\address{Institut f\"ur Mathematik und Wissenschaftliches Rechnen, Universit\"at Graz, Heinrichstrasse 36, 8010 Graz, Austria.}
 \email{charles.beil@uni-graz.at}
\title[A general.\ of cancellative dimer algebras to hyperbolic surfaces]{A generalization of cancellative dimer algebras to hyperbolic surfaces}
 \date{}

\maketitle

\begin{abstract}
We study a new class of quiver algebras on surfaces, called `geodesic ghor algebras'. 
These algebras generalize cancellative dimer algebras on a torus to higher genus surfaces, where the relations come from perfect matchings rather than a potential. 
Although cancellative dimer algebras on a torus are noncommutative crepant resolutions, the center of any dimer algebra on a higher genus surface is just the polynomial ring in one variable, and so the center and surface are unrelated.
In contrast, we establish a rich interplay between the central geometry of geodesic ghor algebras and the topology of the surface in which they are embedded.
Furthermore, we show that noetherian central localizations of such algebras are endomorphism rings of modules over their centers.
\end{abstract}

\section{Introduction}

Cancellative dimer algebras on a torus have been extensively studied in the contexts of noncommutative resolutions, Calabi-Yau algebras, and stability conditions, e.g., \cite{Br,D,B2}.
It is well known that every cancellative dimer algebra on a torus is a noncommutative crepant resolution, and every three dimensional affine toric Gorenstein singularity admits a noncommutative crepant resolution given by such a dimer algebra.
However, if the dimer algebra is on a surface of genus $g \geq 2$, then these nice properties disappear: the center of such a dimer algebra is simply the polynomial ring in one variable, and so there can be no interesting interactions
between the topology of the surface and the algebras central geometry and representation theory.

In this article we consider special quotients of dimer algebras, called `ghor algebras'.
A ghor algebra is a quiver algebra whose quiver embeds in a surface, with relations determined by the perfect matchings of its quiver (the precise definition is given in Section \ref{results}).
Ghor algebras were introduced in \cite{B1,B4} to study non-cancellative dimer algebras on a torus.\footnote{In \cite{B4}, we called ghor algebras `homotopy algebras' because their relations are homotopy relations on the paths in the quiver when the surface is a torus.
However, in the higher genus case homologous cycles are also identified (see Theorem \ref{iff} below), and therefore the name `homotopy' is less suitable for general surfaces.  The word `ghor' is Klingon for surface.}
We introduce a special property of certain ghor algebras called `geodesic' in Section \ref{results}; on a torus, a ghor algebra is geodesic if and only if it is a cancellative dimer algebra.
On higher genus surfaces, we will show that noetherian localizations of geodesic ghor algebras remain endomorphism rings of modules over their centers, but new features arise.
The purpose of this article is to show that geodesic ghor algebras exhibit a rich interplay between their algebraic properties, central geometry, and the topology of the surface in which they are embedded.

In the following, let $\Sigma$ be a surface obtained from a regular $2N$-gon $P$ by identifying the opposite sides, and vertices, of $P$.
Our first main result gives an explicit description of the center of a geodesic ghor algebra $A = kQ/\operatorname{ker}\eta$ on $\Sigma$.

\begin{Theorem A} (Corollary \ref{injective}.)
The center of a geodesic ghor algebra $A$ on $\Sigma$ is isomorphic to the ring
\begin{equation} \label{R def}
R = k\left[ \cap_{i \in Q_0} \bar{\tau}(e_iAe_i) \right],
\end{equation}
where $\bar{\tau}$ is a particular algebra homomorphism to the polynomial ring generated by the perfect matchings of $Q$ (see Section \ref{results}).
\end{Theorem A}

In Section \ref{ncg} we will show that if $\Sigma$ is hyperbolic, then the centers of geodesic ghor algebras on $\Sigma$ are usually nonnoetherian, in stark contrast to the case where $\Sigma$ is a torus.
We can nevertheless view such a center as the coordinate ring on a geometric space, using the framework of nonnoetherian geometry introduced in \cite{B3}.
In short, the geometry of a nonnoetherian coordinate ring of finite Krull dimension looks just like a finite type algebraic variety, except that it has some positive dimensional closed points (that is, positive dimensional subvarieties with no interior points).
The original motivation for this framework was to make sense of the central geometry of non-cancellative dimer algebras on a torus (see \cite{B2,B5}), and is based on the following definition.

\begin{Definition} \label{depiction definition} \rm{
A \textit{depiction} of an integral domain $k$-algebra $R$ is a finitely generated overring $S$ such that the morphism
$$\operatorname{Spec}S \to \operatorname{Spec}R, \ \ \ \mathfrak{q} \mapsto \mathfrak{q} \cap R,$$
 is surjective, and
\begin{equation} \label{U}
U_{S/R} := \{ \mathfrak{n} \in \operatorname{Max}S \ | \ R_{\mathfrak{n} \cap R } = S_{\mathfrak{n}} \} = \{ \mathfrak{n} \in \operatorname{Max}S \ | \ R_{\mathfrak{n} \cap R} \text{ is noetherian} \} \not = \emptyset.
\end{equation}
}\end{Definition}

For example, the algebra $S = k[x,y]$ is a depiction of its nonnoetherian subalgebra $R = k + xS$.
We thus view $\operatorname{Max}R$ as the variety $\operatorname{Max}S = \mathbb{A}^2_k$, except that the line $\{ x = 0 \}$ is identified as a $1$-dimensional (closed) point of $\operatorname{Max}R$.
In particular, the complement $\{x \not = 0\} \subset \mathbb{A}^2_k$ is the `noetherian locus' $U_{S/(k+xS)}$ \cite[Proposition 2.8]{B3}.
\newpage
In Sections \ref{ncg} and \ref{Krull section} we show the following.

\begin{Theorem B} (Theorems \ref{depiction theorem}, \ref{KZA}.)
Let $A$ be a geodesic ghor algebra with center $R$ on the surface $\Sigma$.
Consider the ring generated by the union of the $\bar{\tau}$-images,
\begin{equation} \label{S def}
S := k\left[ \cup_{i \in Q_0} \bar{\tau}(e_iAe_i) \right];
\end{equation}
we call $S$ the `cycle algebra' of $A$.
The following holds.
\begin{enumerate}
\item If the center $R$ is nonnoetherian, then $R$ is depicted by the cycle algebra $S$.
\item The Krull dimensions of the center $R$ and cycle algebra $S$ coincide, and equal one plus half the number of sides of the fundamental polygon $P$:
\begin{equation*}
\operatorname{dim}R = \operatorname{dim}S = N + 1.
\end{equation*}
In particular, if $\Sigma$ is a smooth genus $g \geq 0$ surface, then
$$\operatorname{dim}R = \operatorname{rank}H_1(\Sigma) + 1 = 2g + 1.$$
\end{enumerate}
\end{Theorem B}

Finally, in Section \ref{endo section} we show that the noetherian locus $U_{S/R} \subset \operatorname{Max}S$ is intimately related to the endomorphism ring structure of $A$.

\begin{Theorem C} (Theorem \ref{endo}.)
Let $A$ be a geodesic ghor algebra with center $R$ on the surface $\Sigma$.
At each point $\mathfrak{m} \in \operatorname{Max}R$ for which the localization $R_{\mathfrak{m}}$ is noetherian, the localization $A_{\mathfrak{m}} := A \otimes_R R_{\mathfrak{m}}$ is an endomorphism ring over its center: for each $i \in Q_0$, we have
$$A_{\mathfrak{m}} \cong \operatorname{End}_{R_{\mathfrak{m}}}(A_{\mathfrak{m}}e_i).$$
The locus of such points forms the noetherian locus, and lifts to an open dense subset of the algebraic variety $\operatorname{Max}S$.
\end{Theorem C}

\begin{figure}
$$\begin{array}{ccccc}
\includegraphics[scale=.63]{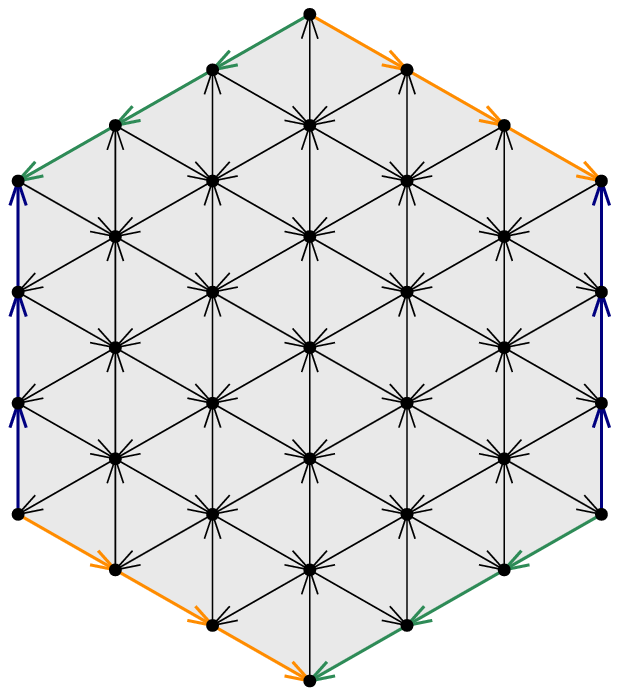}
& \ &
\includegraphics[scale=.63]{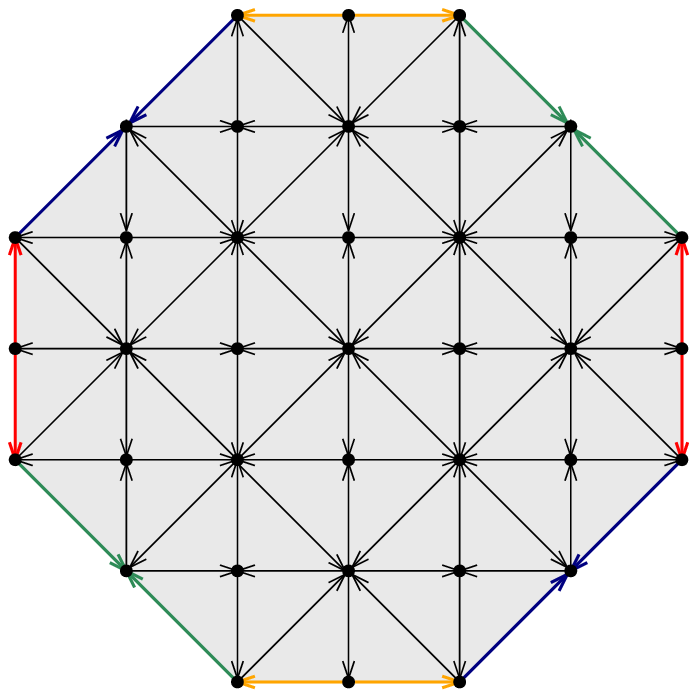}
& \ &
\includegraphics[scale=.63]{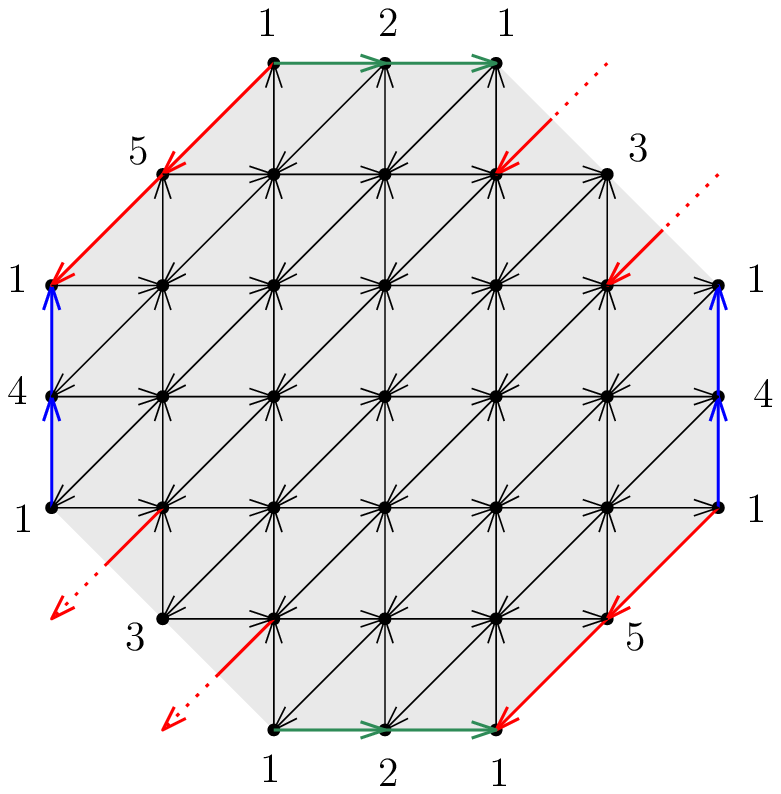}
\\ 
(i) & & (ii) & & (iii)
\end{array}$$
\caption{Examples of geodesic ghor algebras.
Opposite sides of the polygons are identified.
The ghor algebra (i) is on a pinched torus, and the ghor algebras (ii) and (iii) are on a smooth genus $2$ surface.
The centers of (i) and (ii) are given explicitly in \cite[Sections 3.2, 3.3]{BB}.}
\label{examples}
\end{figure}

\begin{figure}
\includegraphics[width=10cm]{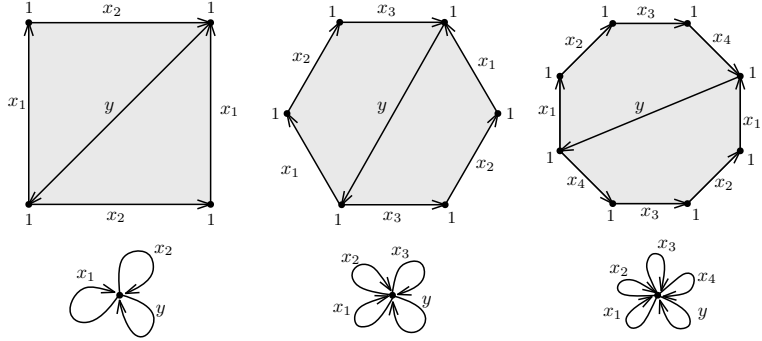}
\caption{The polynomial ghor algebras $A = k[x_1,x_2,y]$, $A = k[x_1,x_2,x_3,y]$, and $A = k[x_1,x_2,x_3,x_4,y]$.}
\label{fig:polynomial}
\end{figure}

\begin{figure}
\includegraphics[width=10cm]{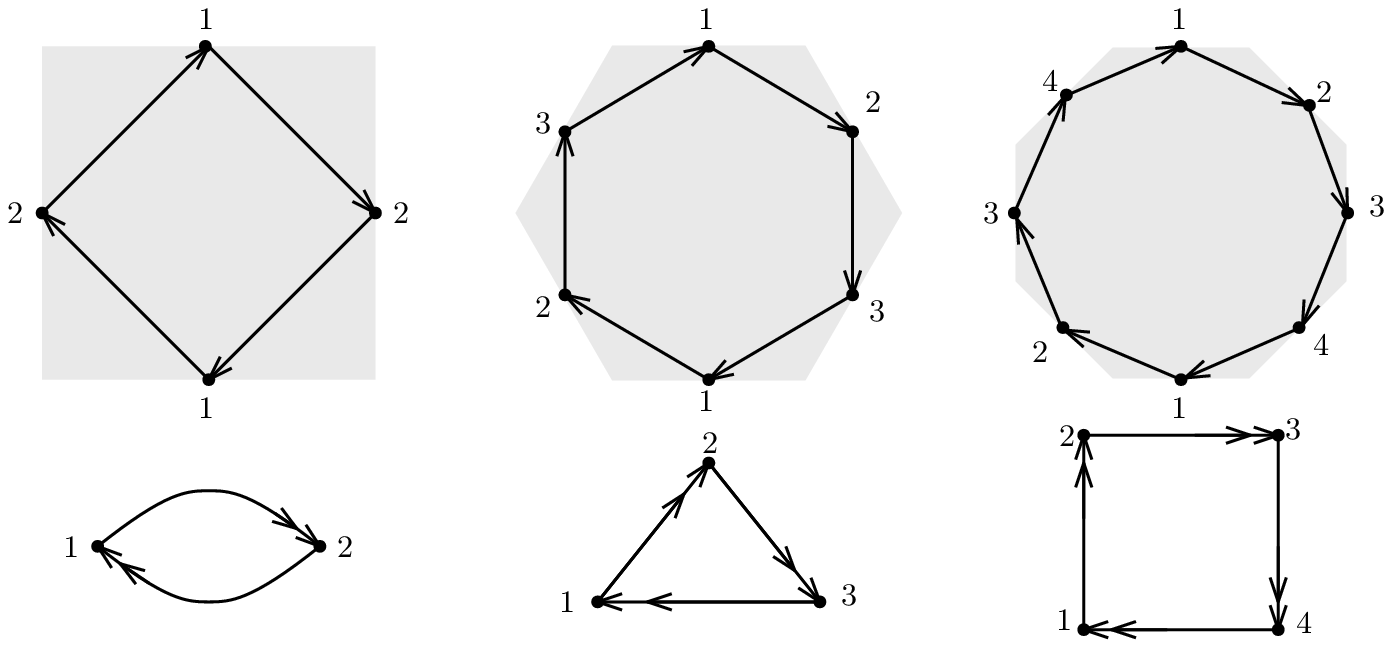}
\caption{A generalization of the conifold dimer algebra on a torus (shown on the left) to geodesic ghor algebras on genus $g$ surfaces.
The grey polygon region is the fundamental domain $P$ of the surface.
These geodesic ghor algebras are noetherian, satisfy $R = S$ and $\mathcal{S} = \mathcal{P}$, and are the only known noetherian ghor algebras with more than one vertex in the case $g \geq 2$.}
\label{conifold}
\end{figure}

\section{Preliminary definitions} \label{results}

\begin{Notation} \rm{
Throughout, $k$ is an uncountable algebraically closed field.
We denote by $\operatorname{Spec}S$ and $\operatorname{Max}S$ the prime ideal spectrum (or scheme) and maximal ideal spectrum (or affine variety) of $S$, respectively.
Given a quiver $Q$, we denote by $kQ$ the path algebra of $Q$; by $Q_{\ell}$ the paths of length $\ell$; by $\operatorname{t}, \operatorname{h}: Q_1 \to Q_0$ the tail and head maps; and by $e_i$ the idempotent at vertex $i \in Q_0$.
By \textit{cyclic subpath} of a path $p$, we mean a subpath of $p$ that is a nontrivial cycle.
We denote by $[n]$ the set $\{1, 2, \ldots, n\}$.
Finally, we denote by $e_{ij} \in M_n(k)$ the $n \times n$ matrix with a $1$ in the $ij$-th slot and zeros elsewhere.
}\end{Notation}

In this article we consider surfaces $\Sigma$ that are obtained from a regular convex $2N$-gon $P$, $N \geq 2$, by identifying the opposite sides, and vertices, of $P$.
This class of surfaces includes all smooth orientable compact closed connected genus $g \geq 1$ surfaces.
Specifically,
\begin{itemize}
 \item if $P$ is a $4g$-gon, then $\Sigma$ is a smooth genus $g$ surface; and
 \item if $P$ is a $2(2g+1)$-gon, then $\Sigma$ is a genus $g$ surface with a pinched point.
\end{itemize}
The polygon $P$ is then a fundamental polygon for $\Sigma$.

If $N = 2$, then $\Sigma$ is a torus, and the covering space of $\Sigma$ is the plane $\mathbb{R}^2$.
For $N \geq 3$, the covering space of $\Sigma$ is the hyperbolic plane $\mathbb{H}^2$.
The hyperbolic plane may be represented by the interior of the unit disc in $\mathbb{R}^2$, where straight lines in $\mathbb{H}^2$ are segments of circles that meet the boundary of the disc orthogonally.
In the covering, the hyperbolic plane is tiled with regular $2N$-gons, with $2N$ such polygons meeting at each vertex.
In this case, $\Sigma$ is said to be a hyperbolic surface.

\begin{Definition} \label{definitions} \rm{ 
A \textit{dimer quiver} on $\Sigma$ is a quiver $Q$ whose underlying graph $\overbar{Q}$ embeds in $\Sigma$, such that each connected component of $\Sigma \setminus \overbar{Q}$ is simply connected and bounded by an oriented cycle, called a \textit{unit cycle}.

\begin{itemize}
 \item[-] A \textit{perfect matching} of a dimer quiver $Q$ is a set of arrows $x \subset Q_1$ such that each unit cycle contains precisely one arrow in $x$.
 \item[-] A perfect matching $x$ is called \textit{simple} if $Q \setminus x$ contains a cycle that passes through each vertex of $Q$ (equivalently, $Q \setminus x$ supports a simple $kQ$-module of dimension vector $(1,1, \ldots, 1)$).
\end{itemize}

Denote by $\mathcal{P}$ and $\mathcal{S}$ the set of perfect and simple matchings of $Q$, respectively.
We will consider the polynomial rings $k[\mathcal{P}]$ and $k[\mathcal{S}]$ generated by these matchings.
\textit{Throughout, we assume that each arrow of $Q$ is contained in a perfect matching.}
}\end{Definition}

Let $Q$ be a dimer quiver on a surface $\Sigma$. 
In the following, we define the ghor algebra and dimer algebra of $Q$.

\begin{Definition} \rm{\ \\
\indent $\bullet$ Consider the algebra homomorphism
$$\eta: kQ \to M_{|Q_0|}\left(k[\mathcal{P}]\right)$$
defined on the vertices $i \in Q_0$ and arrows $a \in Q_1$ by
$$\eta(e_i) = e_{ii}, \ \ \ \ \ \ \ \ \eta(a) = e_{\operatorname{h}(a),\operatorname{t}(a)} \prod_{x \in \mathcal{P}: \, x \ni a} x,$$
and extended multiplicatively and $k$-linearly to $kQ$.
The \textit{ghor algebra} of $Q$ is the quotient
$$A := kQ/\operatorname{ker}\eta.$$

We will also consider the algebra homomorphism
$$\tau: kQ \to M_{|Q_0|}\left(k[\mathcal{S}]\right)$$
similarly defined on the vertices $i \in Q_0$ and arrows $a \in Q_1$ by
$$\tau(e_i) = e_{ii}, \ \ \ \ \ \ \ \ \tau(a) = e_{\operatorname{h}(a),\operatorname{t}(a)} \prod_{x \in \mathcal{S} : \, x \ni a} x.$$
That is, $\tau$ is obtained from $\eta$ by setting all the non-simple variables equal to $1$.

$\bullet$ The \textit{dimer algebra} of $Q$ is the quotient of $kQ$ by the ideal
$$I = \left\langle p - q \ | \ \exists a \in Q_1 \text{ such that } pa, qa \text{ are unit cycles} \right\rangle \subset kQ,$$
where $p,q$ are paths.
A dimer algebra $kQ/I$ is said to be \textit{non-cancellative} if there are paths $p,q,r \in kQ/I$ such that $p \not = q$, but
$$rp = rq \not = 0 \ \ \text{ or } \ \ pr = qr \not = 0;$$
otherwise $kQ/I$ is \textit{cancellative}. 
}\end{Definition}

A ghor algebra $A = kQ/\operatorname{ker}\eta$ is the quotient of the dimer algebra $kQ/I$ since $I \subseteq \operatorname{ker}\eta$: if $pa, qa$ are unit cycles with $a \in Q_1$, then
$$\eta(p) = e_{\operatorname{h}(p),\operatorname{t}(p)} \prod_{x \in \mathcal{P}: \, x \not \ni a} x = \eta(q).$$ 
Dimer algebras on non-torus surfaces have been considered in the context of, for example, cluster categories \cite{BKM,K}, Belyi maps \cite{BGH}, and gauge theories \cite{FGU,FH}.


\begin{Notation} \rm{
Let $\pi: \Sigma^+ \to \Sigma$ be the projection from the covering space $\Sigma^+$ (here, $\mathbb{R}^2$ or the hyperbolic plane $\mathbb{H}^2$) to the surface $\Sigma$.
Denote by $Q^+ := \pi^{-1}(Q) \subset \Sigma^+$ the (infinite) covering quiver of $Q$, and by $p^+$ the lift of a path $p$ to $Q^+$.
}\end{Notation}

\begin{Definition} \rm{
We say a homotopy $H: \overline{Q} \times [0,1] \to \Sigma$ of the underlying graph $\overline{Q}$ of $Q$ is \textit{dimer-preserving} if for all $t \in [0,1]$, $H(-,t): \overline{Q} \to \Sigma$ is an embedding, and each connected component of $\Sigma \setminus H(\overline{Q},t)$ is simply connected and bounded by unit cycle.
We will often omit the overline and simply write $H: Q \times [0,1] \to \Sigma$.

Fix a tiling of the covering space of $\Sigma$ by fundamental polygons $P$, and label the sides $1, 2, \ldots, {2N}$ of $P$ in counterclockwise order, with indices taken modulo $2N$.
Let $H: Q \times [0,1] \to \Sigma$ be a dimer-preserving homotopy for which all the vertices of $H(Q,1)$ lift to vertices that lie in the interior of $P$, and no arrow intersects a corner vertex $v$ of $P$:
\begin{equation} \label{H(Q,1)}
H(Q_0,1) \cap \pi(\partial P) = \emptyset, \ \ \ \ \ H(Q_1,1) \cap \pi(v) = \emptyset.
\end{equation}
We say a path $p$ \textit{transversely intersects} side $k$ of $P$ with respect to $H$ if $p$ intersects $k$ transversely in $H(Q,1)$.

Given a cycle $c$, we define the \textit{class} of $c$ to be
$$[c] := \sum_{k \in [N]} (n_k - n_{k+N}) (\delta_{k \ell})_{\ell} \in \mathbb{Z}^N,$$
where $n_k$, $k \in [2N]$, is the number of times $c$ transversely intersects side $k$ of $P$.
If $\Sigma$ is smooth (that is, if $N$ is even), then $[c]$ is the homology class of $c$ in $H_1(\Sigma) := H_1(\Sigma, \mathbb{Z})$.
}\end{Definition}

We introduce the following special class of ghor algebras which generalizes cancellative dimer algebras on a torus.

\begin{Definition} \label{geocom} \rm{ \
\begin{itemize}
 \item A cycle $p \in A$ is \textit{geodesic} if the lift to $Q^+$ of each cyclic permutation of each representative of $p$ does not have a cyclic subpath.
 \item Two cycles are \textit{parallel} if they do not transversely intersect.
 \item A ghor algebra is \textit{geodesic} if for each $k \in [2N]$, there is a geodesic cycle $\gamma_k$ with class
\begin{equation} \label{geodesic cycle def}
[\gamma_k] =  (\delta_{k \ell} - \delta_{k+N,\ell})_{\ell \in [N]} \in \mathbb{Z}^N,
\end{equation}
with indices modulo $2N$, and a set of pairwise parallel geodesic cycles
$$\{c_i \in e_i kQ e_i\}_{i \in Q_0}$$
such that $c_{\operatorname{t}(\gamma_k)} = \gamma_k$.
\end{itemize}
}\end{Definition}

\begin{Remark} \rm{
We note that if $\Sigma$ is hyperbolic, then for fixed $k \in [N]$, the parallel geodesic cycles $c_i$ will in general be in different homology classes; in particular, $[c_i]$ need not equal $[\gamma_k]$.
However, if $\Sigma$ is flat (that is, $N = 2$), then for fixed $k \in [2]$, the parallel geodesic cycles $c_i$ may be chosen to be in the same homology class.
}\end{Remark}

\begin{Remark} \rm{
Consider the ghor algebra $A$ with quiver $Q$ on a genus $2$ surface given in Figure \ref{examples}.ii.
Label the sides of $P$ by $1, \ldots, 8$, starting with the top side and continuing counterclockwise around $P$.
Observe that $A$ is geodesic.
However, there is no geodesic cycle at the vertex in the center of $P$ that intersects both sides $1$ and $7$ transversely.
Thus, it is too restrictive to require that there is a geodesic cycle at each vertex in each homology class of $\Sigma$; see also Remark \ref{curve}.
}\end{Remark}

For $i,j \in Q_0$, consider the $k$-linear maps
$$\bar{\eta}: e_jkQe_i \to k[\mathcal{P}] \ \ \ \ \text{ and } \ \ \ \ \bar{\tau}: e_jkQe_i \to k[\mathcal{S}]$$
defined by sending $p \in e_jkQe_i$ to the single nonzero matrix entry of $\eta(p)$ and $\tau(p)$ respectively; that is,
$$\eta(p) = \bar{\eta}(p)e_{ji} \ \ \ \ \text{ and } \ \ \ \ \tau(p) = \bar{\tau}(p)e_{ji}.$$
These maps define multiplicative labelings of the paths of $Q$.
We will often write $\overbar{p}$ for $\bar{\eta}(p)$ or $\bar{\tau}(p)$.

An important monomial is the $\bar{\eta}$- and $\bar{\tau}$-images of each unit cycle in $Q$, namely
$$\sigma_{\mathcal{P}} := \prod_{x \in \mathcal{P}}x \ \ \ \ \text{ and } \ \ \ \ \sigma_{\mathcal{S}} := \prod_{x \in \mathcal{S}}x.$$
We will omit the subscript $\mathcal{P}$ or $\mathcal{S}$ if it is clear from the context.

Since unit cycles $\sigma_i$ are contractible curves on $\Sigma$, the topology of $\Sigma$ is more closely reflected in the quotient rings
\begin{equation} \label{quotient}
k[\mathcal{P}]/(\sigma_{\mathcal{P}} -1) \ \ \ \ \text{ and } \ \ \ \ k[\mathcal{S}]/(\sigma_{\mathcal{S}} -1).
\end{equation}
If polynomials $g,h$ are equal in the quotient, that is, if there is an $\ell \in \mathbb{Z}$ such that $g = h \sigma^{\ell}$, then we will write
$$g \stackrel{\sigma}{=} h.$$

\begin{Notation} \rm{
Given paths $p,q$, we write $\bar{\eta}(p) \mid \bar{\eta}(q)$ (resp.\ $\bar{\tau}(p) \mid \bar{\tau}(p)$) if $\bar{\eta}(p)$ divides $\bar{\eta}(q)$ in $k[\mathcal{P}]$ (resp.\ $\bar{\tau}(p)$ divides $\bar{\tau}(q)$ in $k[\mathcal{S}]$).
}\end{Notation}

\section{A bridge from topology to algebra: subdivisions and simple matchings}

Let $A = kQ/ \ker \eta$ be a ghor algebra.

\begin{Lemma} \label{Z}
\
\begin{itemize}
 \item[(i)] If $p^+$ is a cycle in $Q^+$, that is, $p = \pi(p^+)$ is a contractible cycle, then
$$\bar{\eta}(p) \stackrel{\sigma}{=} 1 \ \ \ \text{ and } \ \ \ \bar{\tau}(p) \stackrel{\sigma}{=} 1.$$
 \item[(ii)] If $p^+$, $q^+$ are paths in $Q^+$ satisfying
\begin{equation} \label{t and h}
\operatorname{t}(p^+) = \operatorname{t}(q^+) \ \ \ \text{ and } \ \ \ \operatorname{h}(p^+) = \operatorname{h}(q^+),
\end{equation}
then
$$\bar{\eta}(p) \stackrel{\sigma}{=} \bar{\eta}(q) \ \ \ \text{ and } \ \ \ \bar{\tau}(p) \stackrel{\sigma}{=} \bar{\tau}(q).$$
\end{itemize}
\end{Lemma}

\begin{proof}
(i) We proceed by induction on the number of faces contained in the region $\mathcal{R}_p$ bounded by $p^+$.
Factor $p$ into a minimum number of subpaths
$$p = p_m \cdots p_2p_1,$$
where each $p_j$ is a subpath of a unit cycle.
For each $j \in [m]$, let $r_j$ be the path for which $r_jp_j$ is a unit cycle and $r_j^+$ lies in $\mathcal{R}_p$.
Since $m$ is minimum, the concatenation
$$r = r_1 \cdots r_{m-1}r_m$$
is a cycle whose lift $r^+$ lies in the region $\mathcal{R}_p$.

Without loss of generality, we may assume that at least one $p_j$ is not a unit cycle, hence at least one $r_j$ is not a vertex.
Thus, by the induction hypothesis, there is an $n \geq 0$ such that
$$\bar{\eta}(r) = \sigma_{\mathcal{P}}^n.$$
Furthermore,
$$\bar{\eta}(r) \bar{\eta}(p) = \prod_j \bar{\eta}(r_j) \prod_j \bar{\eta}(p_j) = \prod_j \bar{\eta}(r_jp_j) = \sigma_{\mathcal{P}}^m.$$
Therefore
$$\bar{\eta}(p) = \sigma_{\mathcal{P}}^{m-n}.$$
Similarly, $\bar{\tau}(p) = \sigma_{\mathcal{S}}^{m-n}$.

(ii) Suppose $p^+$, $q^+$ are paths in $Q^+$ satisfying (\ref{t and h}).
Let $r^+$ be a path in $Q^+$ from $\operatorname{h}(p^+)$ to $\operatorname{t}(p^+)$.
Then by Claim (i), there is an $m, n \geq 0$ such that
$$\bar{\eta}(r)\bar{\eta}(p) = \bar{\eta}(rp) = \sigma_{\mathcal{P}}^m \ \ \ \text{ and } \ \ \ \bar{\eta}(r)\bar{\eta}(q) = \bar{\eta}(rq) = \sigma_{\mathcal{P}}^n.$$
Whence
$$\bar{\eta}(p) = \bar{\eta}(q) \sigma_{\mathcal{P}}^{m-n}.$$
Similarly, $\bar{\tau}(p) = \bar{\tau}(q) \sigma_{\mathcal{S}}^{m-n}$.
\end{proof}

\begin{Lemma}
In the quotient rings (\ref{quotient}), every path $p \in e_jAe_i$ has an inverse $q \in e_iAe_j$:
$$\overbar{p} \overbar{q} \stackrel{\sigma}{=} \overbar{q} \overbar{p} \stackrel{\sigma}{=} 1.$$
We will write $\overbar{p}^{-1} := \overbar{q}$.
\end{Lemma}

\begin{proof}
Fix a path $p$, and let $q$ be any path satisfying
$$\operatorname{t}(q^+) = \operatorname{h}(p^+) \ \ \ \text{ and } \ \ \ \operatorname{h}(q^+) = \operatorname{t}(p^+).$$
Then $(pq)^+$ is a cycle in $Q^+$.
Consequently, $\overbar{p} \overbar{q} = \overbar{pq} \stackrel{\sigma}{=} 1$ by Lemma \ref{Z}.i.
\end{proof}

In the following we describe an algebraic feature of ghor algebras that is a consequence of the curvature of $\Sigma$.

\begin{Remark} \rm{
If $A$ is a geodesic ghor algebra on a torus, and $i,j \in Q_0^+$ are distinct vertices, then there is always a path $p^+$ from $i$ to $j$ in $Q^+$ such that $\sigma \nmid \overbar{p}$ \cite[Proposition 4.20.iii]{B1}.
However, this implication no longer holds if the surface $\Sigma$ is not flat.
Indeed, in this case there are always distinct vertices $i, j \in Q_0^+$ for which every path $p^+$ from $i$ to $j$ satisfies $\overbar{p} = \sigma^{\ell}$ for some $\ell \geq 1$; see \cite[Remark 2.7]{BB}.
}\end{Remark}

\begin{Proposition} \label{how}
Suppose $p,q$ are cycles in the same class, $[p] = [q]$.
Then $\overbar{p} \stackrel{\sigma}{=} \overbar{q}$.
\end{Proposition}

\begin{proof}
Fix a tiling of the covering space by fundamental polygons $P$, and a dimer-preserving homotopy $H(Q,1) \hookrightarrow \Sigma$ satisfying (\ref{H(Q,1)}).
For a cycle $p$ in $Q$ (not modulo $\ker \eta$), denote by $\operatorname{ord}(p)$ the sequence of sides of $P$ intersected by the arrow subpaths of $p$ in $H(Q,1)$ in order,
$$\operatorname{ord}(p) := ( j(1), j(2), \ldots, j(m) ).$$

If $p$ and $q$ both do not intersect the boundary $\pi(\partial P)$, then by Lemma \ref{Z}.ii,
$$\overbar{p} \stackrel{\sigma}{=} 1 \stackrel{\sigma}{=} \overbar{q}.$$

So suppose $p$ (and thus $q$) intersects $\pi(\partial P)$.
Consider the ordering of $p$ and $q$ with respect to $H(Q,1)$,
$$\operatorname{ord}(p) = ( j(1), j(2), \ldots, j(m) ) \ \ \ \text{ and } \ \ \ \operatorname{ord}(q) = ( j'(1), j'(2), \ldots, j'(n) ).$$
Without loss of generality, we may assume $m \geq n$.

Factor $p$ into paths $p = d_2ad_1$, where $d_2$ has minimal length such that $a$ is an arrow that intersects the boundary $\pi(\partial P)$.
Consider the cyclic permutation $p' := ad_1d_2$.
Then $\overbar{p}' = \overbar{p}$.
Furthermore, since $d_2$ has minimal length, we have $\operatorname{ord}(p) = \operatorname{ord}(p')$.
It therefore suffices to assume that the leftmost arrow subpath of $p$ intersects $\pi(\partial P)$.
Under this assumption, we may factor $p$ into paths
$$p = a_m p_{m} \cdots a_2p_2a_1p_1,$$
where each $a_k \in Q_1$ is an arrow that intersects $\pi(\partial P)$, and each $p_k \in Q_{\geq 0}$ is a path that does not intersect $\pi(\partial P)$.

Fix $i \in Q_0$.
For each $k \in [m]$ (modulo $m$), consider a path\footnote{By `path', we always mean a possibly trivial path unless stated otherwise.} $s_{k-1}$ from $\operatorname{h}(a_{k-1})$ to $i$, and a path $t_k$ from $i$ to $\operatorname{t}(a_k)$, both of which lift to paths that lie in the interior of $P$; see Figure \ref{prop-3-5}.
By Lemma \ref{Z}.ii, we have
$$ \overbar{t}_k \overbar{s}_{k-1} \stackrel{\sigma}{=} \overbar{p}_k.$$
Thus,
\begin{equation} \label{k=1}
\overbar{p} = \prod_{k = 1}^{m} \overbar{a}_k \overbar{p}_k \stackrel{\sigma}{=} \prod_{k =1}^{m} \overbar{a}_k \overbar{t}_k \overbar{s}_{k-1}.
\end{equation}

Set 
\begin{equation} \label{r_k}
r_k := s_ka_kt_k \in e_ikQe_i.
\end{equation}
Since $p$ and $q$ are in the same class $u := [p] = [q] \in \mathbb{Z}^N$ and $m \geq n$, we can pair off $m - n$ of the $m$ $r_k$ cycles, say $r_{\ell}$, $r_{\ell'}$, such that $a_{\ell}$ and $a_{\ell'}$ intersect the same side of $P$ but in opposite directions.
Consequently, the lift of the concatenation $r_{\ell}r_{\ell'}$ is a cycle in $Q^+$.
Thus $r_{\ell}r_{\ell'} \stackrel{\sigma}{=} 1$, by Lemma \ref{Z}.i.
Therefore, up to a factor of $\sigma$, we may omit these $m-n$ cycles from $\overbar{p}$:
$$\overbar{p} \stackrel{\sigma}{=} \prod_{k =1}^m \overbar{r}_k \stackrel{\sigma}{=} \prod_{\kappa = 1}^n \overbar{r}_{k(\kappa)}.$$
It thus suffices to suppose $m = n$.

Consider a permutation $\varphi$ of $[m-1]$ for which $j(\varphi(k)) = j'(k)$.
The cycle
\begin{equation} \label{r=}
r := t_m r_{\varphi(m-1)}r_{\varphi(m-2)} \cdots r_{\varphi(2)}r_{\varphi(1)} s_m
\end{equation}
then satisfies
\begin{equation} \label{gu}
\operatorname{ord}(r) = \operatorname{ord}(q).
\end{equation}
By (\ref{k=1}) and (\ref{r=}), we have
\begin{equation} \label{r sigma =}
\overbar{r} \stackrel{\sigma}{=} \overbar{p}.
\end{equation}

Let $\alpha$ be path from $\operatorname{t}(q)$ to $\operatorname{t}(p) = \operatorname{t}(r)$, and let $\beta$ be a path from $\operatorname{t}(p)$ to $\operatorname{t}(q)$, such that the lift $(\alpha \beta)^+$ lies in a single fundamental polygon $P$.
In particular, $(\alpha \beta)^+$ is a cycle in $Q^+$.
Thus $\overbar{\alpha \beta} \stackrel{\sigma}{=} 1$, by Lemma \ref{Z}.i.
Furthermore, (\ref{gu}) implies that
$$\operatorname{ord}(r) = \operatorname{ord}(q) = \operatorname{ord}(\alpha q \beta).$$
Whence,
$$\operatorname{t}((\alpha q \beta)^+) = \operatorname{t}(r^+) \ \ \ \text{ and } \ \ \ \operatorname{h}((\alpha q \beta)^+) = \operatorname{h}(r^+).$$
Thus
\begin{equation} \label{q sigma =}
\overbar{r} \stackrel{\sigma}{=} \overbar{\alpha q \beta} = \overbar{q} \overbar{\alpha \beta} \stackrel{\sigma}{=} \overbar{q},
\end{equation}
by Lemma \ref{Z}.ii.
Therefore, by (\ref{r sigma =}) and (\ref{q sigma =}),
$$\overbar{p} \stackrel{\sigma}{=} \overbar{r} \stackrel{\sigma}{=} \overbar{q},$$
which is what we wanted to show.
\end{proof}

\begin{figure}
\includegraphics[width=4.5cm]{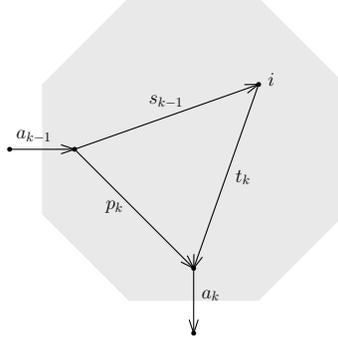}
\caption{Setup for Proposition \ref{how}.}
\label{prop-3-5}
\end{figure}

The following definitions connect the topology of the surface $\Sigma$ with the simple matchings of $Q$, and thus with the algebraic structure of the ghor algebra $A$.

\begin{Definition} \rm{
We call the subquiver given in Figure \ref{figure3}.i a \textit{column}, and the subquiver given in Figure \ref{figure3}.ii a \textit{pillar}.
A \textit{subdivision} of $Q$ is a set $F$ of columns and pillars such that each arrow either (i) lies in the interior of at most one column or pillar; or (ii) belongs to the boundary of a column or pillar.
}\end{Definition}

The following lemma is immediate.

\begin{Lemma} \label{M}
Suppose $A = kQ/\operatorname{ker}\eta$ is geodesic.
Then for each $k \in [2N]$, the set of parallel geodesic cycles $\{c_i\}_{i \in Q_0}$ containing $\gamma_k$ determines a subdivision of $Q$.
\end{Lemma}

\begin{figure}
$$\begin{array}{c}
\xy  0;/r.4pc/:
(-30,2.6)*{\cdot}="1";(-24,2.6)*{\cdot}="2";(-18,2.6)*{\cdot}="3";(-12,2.6)*{\cdot}="4";
(-6,2.6)*{\cdot}="5";
(3,2.6)*{\cdot}="6";(9,2.6)*{\cdot}="7";(15,2.6)*{\cdot}="8";(21,2.6)*{\cdot}="9";
(27,2.6)*{\cdot}="10";
(-27,-2.6)*{\cdot}="11";(-21,-2.6)*{\cdot}="12";(-15,-2.6)*{\cdot}="13";
(-9,-2.6)*{\cdot}="14";(-3,-2.6)*{\cdot}="15";
(6,-2.6)*{\cdot}="16";(12,-2.6)*{\cdot}="17";(18,-2.6)*{\cdot}="18";
(24,-2.6)*{\cdot}="19";(30,-2.6)*{\cdot}="20";
{\ar@[blue]"1";"2"};{\ar@[blue]"2";"3"};{\ar@[blue]"3";"4"};{\ar@[blue]"4";"5"};
{\ar@{.}@[blue]"5";"6"};{\ar@[blue]"6";"7"};{\ar@[blue]"7";"8"};{\ar@[blue]"8";"9"};
{\ar@[blue]"9";"10"};
{\ar@[red]"11";"12"};{\ar@[red]"12";"13"};{\ar@[red]"13";"14"};{\ar@[red]"14";"15"};
{\ar@[red]@{.}"15";"16"};{\ar@[red]"16";"17"};{\ar@[red]"17";"18"};{\ar@[red]"18";"19"};
{\ar@[red]"19";"20"};
{\ar@[brown]"20";"10"};{\ar"10";"19"};{\ar"19";"9"};{\ar"9";"18"};{\ar"18";"8"};{\ar"8";"17"};{\ar"17";"7"};{\ar"7";"16"};{\ar"16";"6"};
{\ar"15";"5"};{\ar"5";"14"};{\ar"14";"4"};{\ar"4";"13"};{\ar"13";"3"};{\ar"3";"12"};
{\ar"12";"2"};{\ar"2";"11"};{\ar@[brown]"11";"1"};
\endxy
\\
\text{a column subquiver}\\
\\
\xy  0;/r.4pc/:
(-15,2.6)*{}="0";(-12,2.6)*{\cdot}="1";(-6,2.6)*{\cdot}="2";
(3,2.6)*{\cdot}="3";(9,2.6)*{\cdot}="4";(15,2.6)*{\cdot}="5";
(-18,0)*{\cdot}="6";
(-15,-2.6)*{\cdot}="7";(-9,-2.6)*{\cdot}="8";(-3,-2.6)*{\cdot}="9";
(6,-2.6)*{\cdot}="10";(12,-2.6)*{\cdot}="11";
(15,-2.6)*{}="12";
(18,0)*{\cdot}="13";
{\ar@[blue]@/^.3pc/@{-}"6";"0"};{\ar@[blue]"0";"1"};{\ar@[blue]"1";"2"};
{\ar@{.}"2";"3"};{\ar@[blue]"3";"4"};{\ar@[blue]"4";"5"};
{\ar@[blue]@/^.3pc/"5";"13"};
{\ar@[red]@/_.3pc/"6";"7"};{\ar@[red]"7";"8"};{\ar@[red]"8";"9"};
{\ar@{.}"9";"10"};{\ar@[red]"10";"11"};{\ar@[red]@{-}"11";"12"};
{\ar@[red]@/_.3pc/"12";"13"};
{\ar@/^.3pc/"13";"5"};{\ar"5";"11"};{\ar"11";"4"};{\ar"4";"10"};{\ar"10";"3"};
{\ar"9";"2"};{\ar"2";"8"};{\ar"8";"1"};{\ar"1";"7"};{\ar@/_.3pc/"7";"6"};
\endxy \\
\text{a pillar subquiver}
\end{array}$$
\caption{A column and pillar of a dimer quiver $Q$.
The black interior arrows are arrows of $Q$; the blue and red boundary arrows are paths of length at least one in $Q$; and each interior cycle is a unit cycle of $Q$.
The leftmost and rightmost brown arrows of the column are identified.
Note that the blue and red bounding paths of the pillar are equal modulo $\ker \eta$.}
\label{figure3}
\end{figure}

\begin{figure}
\includegraphics[width=14cm]{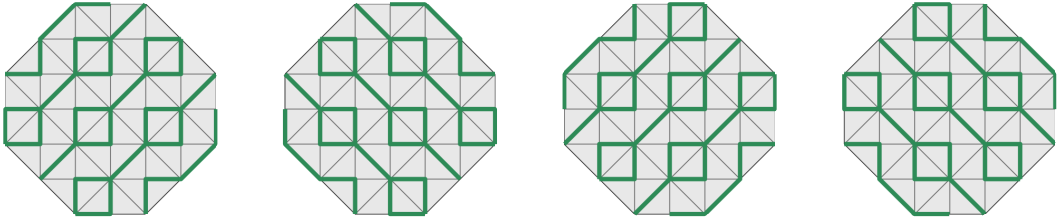}
\caption{The four minimal subdivisions of the geodesic ghor algebra in Figure \ref{examples}.ii.}
\label{four subdivisions}
\end{figure}

\begin{figure}
\includegraphics[width=14cm]{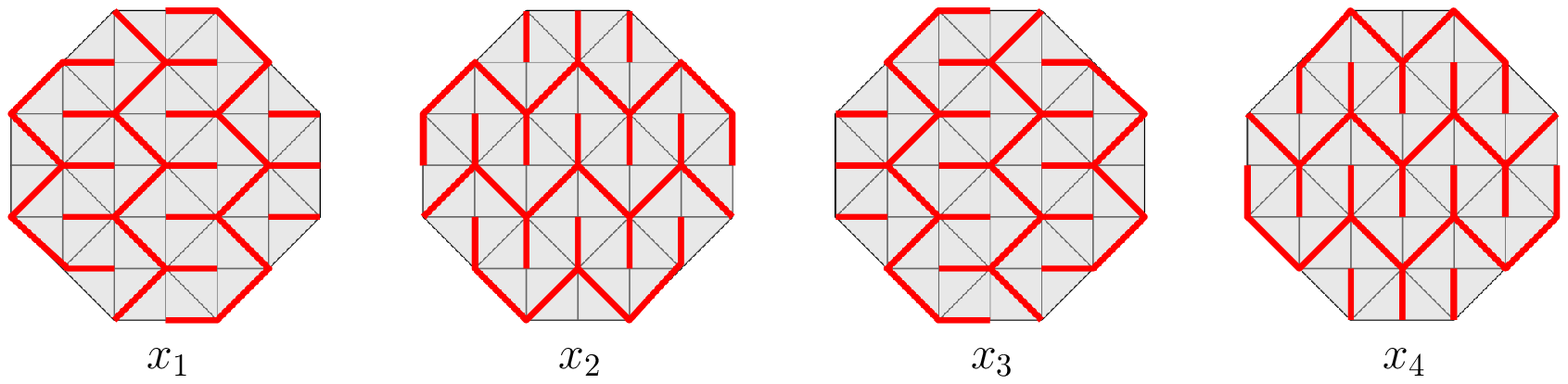}\\
\ \\
\includegraphics[width=14cm]{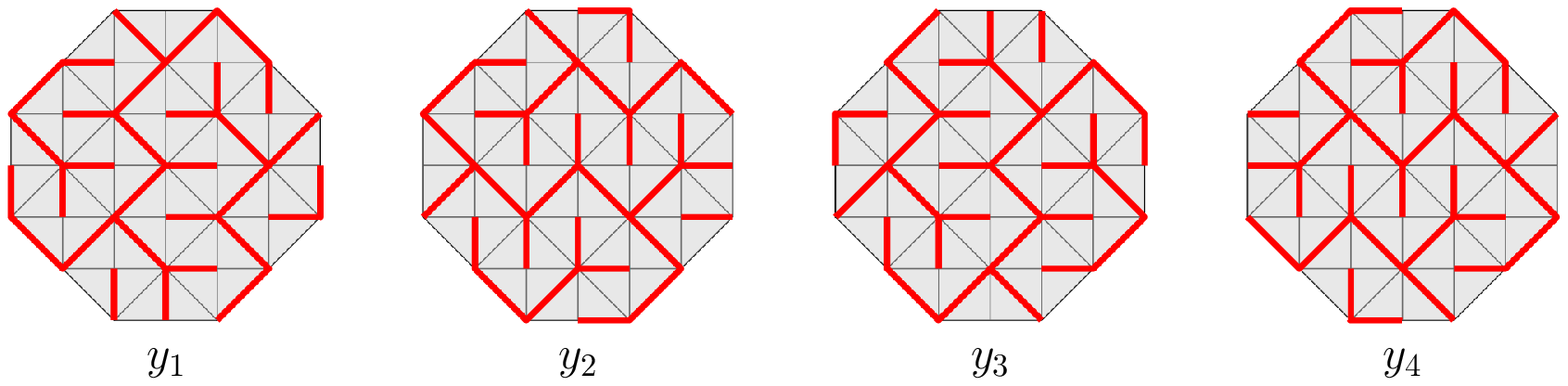}\\
\ \\
\includegraphics[width=14cm]{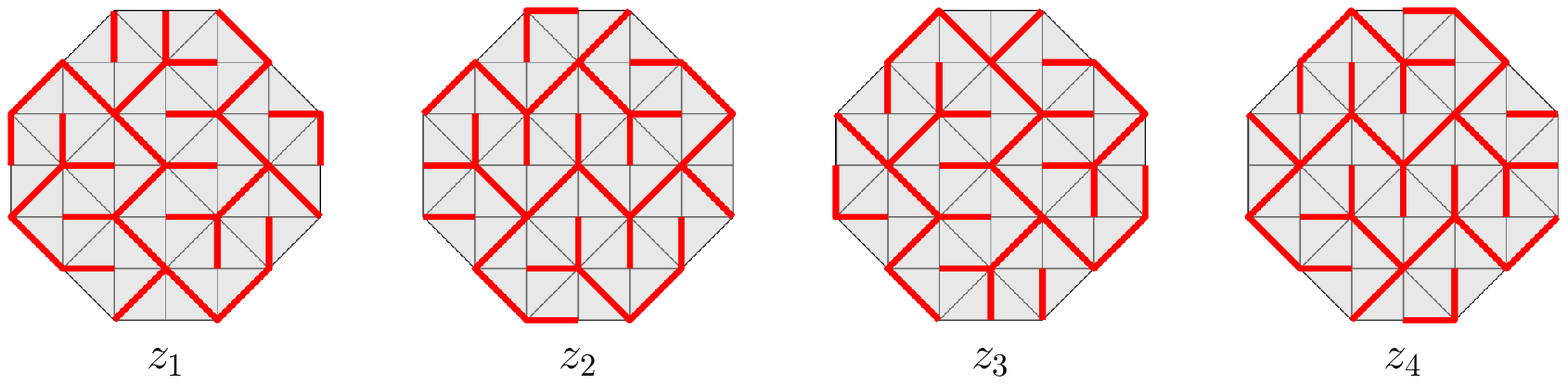}
\caption{The simple matchings of the geodesic ghor algebra in Figure \ref{examples}.ii.}
\label{simple matchings octagon}
\end{figure}

\begin{figure}
\includegraphics[width=12cm]{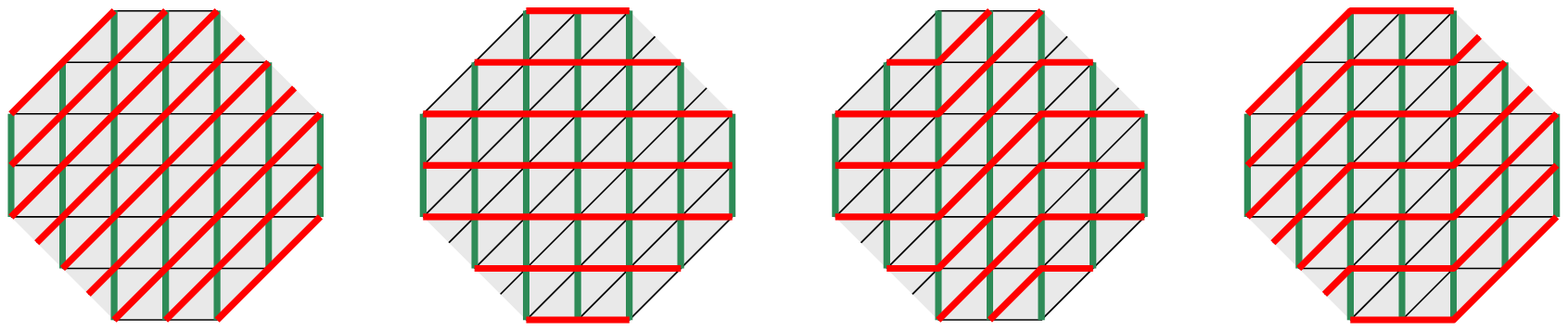}\\ \ \\
\includegraphics[width=12cm]{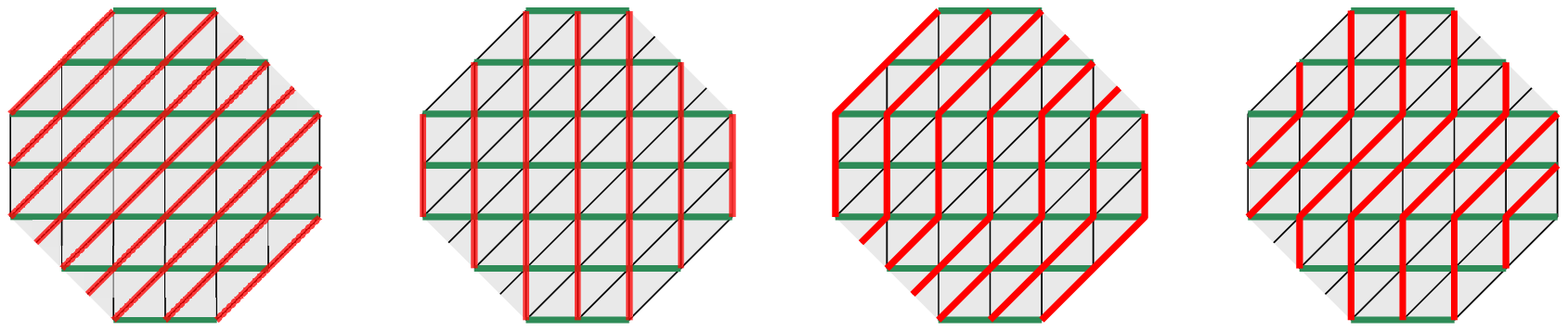}\\ \ \\
\includegraphics[width=12cm]{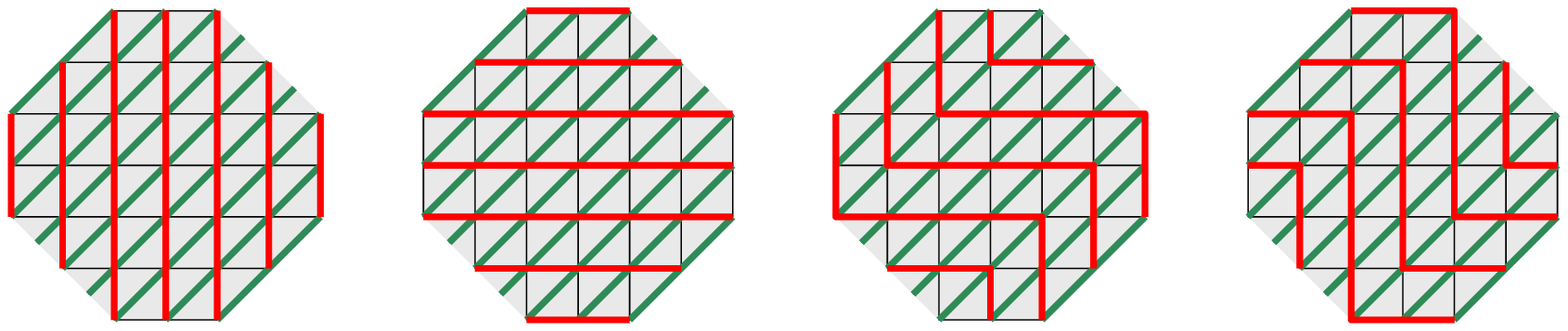}
\caption{The minimal subdivisions and simple matchings of the geodesic ghor algebra in Figure \ref{examples}.iii.
There are three minimal subdivisions (drawn in green) that each yield four simple matchings (drawn in red).
Each minimal subdivision consists only of columns in the respective directions horizontal, vertical, and diagonal.
Each pair of minimal subdivisions yield a common simple matching, and so there are a total of nine simple matchings.}
\label{subdivisionexample}
\end{figure}

\begin{Lemma} \label{columns and pillars} \
Let $p,q$ be paths in $Q$ whose lifts $p^+$, $q^+$ have no cyclic subpaths and bound a region $\mathcal{R}_{p,q}$ which contains no vertices in its interior.
  \begin{enumerate}
    \item If $p$ and $q$ do not intersect, then $p^+$ and $q^+$ bound a column.
    \item Otherwise $p^+$ and $q^+$ bound a union of pillars.
   In particular, if
$$\operatorname{t}(p^+) = \operatorname{t}(q^+) \ \ \ \text{ and } \ \ \ \operatorname{h}(p^+) = \operatorname{h}(q^+) \not = \operatorname{t}(p^+),$$
   then $p \equiv q$.
  \end{enumerate}
Consequently, each subdivision yields a simple matching of $Q$.
\end{Lemma}

\begin{proof}
Follows from \cite[Lemmas 4.12, 4.14, 4.15]{B1}, with the covering space $\mathbb{R}^2$ replaced by the hyperbolic plane $\mathbb{H}^2$ if $\Sigma$ is not flat.
\end{proof}

\begin{Example} \label{sub} \rm{
The geodesic ghor algebra in Figure \ref{examples}.ii embeds in a genus $2$ surface, and admits four minimal subdivisions (that is, subdivisions consisting of a minimal number of columns or pillars) shown in Figure \ref{four subdivisions}.
These four subdivisions yield the twelve simple matchings shown in Figure \ref{simple matchings octagon}.

The geodesic ghor algebra in Figure \ref{examples}.iii also embeds in a genus $2$ surface, and admits three minimal subdivisions shown in Figure \ref{subdivisionexample}.
These three subdivisions consist only of columns, and yield nine distinct simple matchings.
}\end{Example}

\begin{Proposition} \label{arrow simple}
If $A$ is geodesic, then each arrow of $A$ is contained in a simple matching.
\end{Proposition}

\begin{proof}
Suppose $A$ is geodesic, and assume to the contrary that there is an arrow $a \in Q_1$ that is not contained in any simple matching.

Let $\tilde{P} \subset \mathbb{H}^2$ be a fundamental polygon in the cover of $\Sigma$, and let $\iota: \tilde{P} \to \mathbb{R}^2$ be an embedding such that $P := \iota(\tilde{P})$ is a regular $2N$-gon with respect to the standard metric on $\mathbb{R}^2$.
In the following, we will define a dimer-preserving homotopy
$$H: Q \times [0, 2N] \to \Sigma$$
that rotates the arrow $a$ by $2 \pi$.
However, since $\Sigma$ is a compact surface, such a homotopy cannot exist, and thus the arrow $a$ cannot exist, contrary to assumption.

We define $H$ inductively.
Let $H(-,0): Q \to \Sigma$ be a dimer embedding (that is, each connected component of $\Sigma \setminus H(Q,0)$ is simply connected and bounded by a unit cycle), and suppose
$$H: Q \times [0,k] \to \Sigma$$
has been defined.
To define $H: Q \times [k,k+1] \to \Sigma$, set
$$H^+(-,t) := \iota(\tilde{P} \cap \pi^{-1}(H(-,t))).$$
Then $H^+(Q,t) \subset P \subset \mathbb{R}^2$.
We do not require the image $H^+(b,t)$ of an arrow $b$ to be a line segment in $\mathbb{R}^2$, that is, we allow $H$ to `bend' arrows.

Denote by $\vec{\varepsilon}_k$ the unit vector in $\mathbb{R}^2$ based at the midpoint of side $k$ of $P$, orthogonal to side $k$, and pointing to the center of $P$.
Construct a subdivision $F$ from a set of parallel geodesic cycles $\{ c_i \}_{i \in Q_0}$ in the direction $\vec{\varepsilon}_k$, as in Lemma \ref{M}.
We proceed in three steps.

First, let
$$H^+: Q \times [k, k+ \tfrac 13] \to P$$
be a dimer-preserving homotopy that straightens the boundaries of the columns in the subdivision $F$ so that they are parallel to $\vec{\varepsilon}_k$ (in $\mathbb{R}^2$).
Specifically, let $q$ be a boundary path of a column in $F$.
Then $H^+(-,t)$, $t \in [k, k + \tfrac 13]$, `straightens' $q$ so that each vertex subpath of $q$ lies on a line $L \subset \mathbb{R}^2$ that is parallel to $\vec{\varepsilon}_k$, by shrinking the (images of the) arrow subpaths of $q$ that deviate from $L$, and then rotating and translating these arrows so that they lie along $L$.

If $a$ is a subpath of a column boundary path of $F$, then the arrow image $H^+(a,k+\tfrac 13)$ is parallel to $\vec{\varepsilon}_j$.
In this case, set
$$H(Q,t) = H(Q,k+\tfrac 13), \ \ \ \ \text{for } t \in [k+ \tfrac 13, k+1].$$
If $a$ is an interior arrow of a column, then $a$ belongs to a simple matching, contrary to assumption.
So suppose $a$ is a subpath of either a pillar boundary in $F$, or an even arrow in the interior of a pillar.

Let $L_{\operatorname{t}(a)}$, $L_{\operatorname{h}(a)}$ be lines in $\mathbb{R}^2$ that are parallel to $\vec{\varepsilon}_k$, and contain the respective vertices $\operatorname{t}(a)$, $\operatorname{h}(a)$.
Since the columns in $F$ have all been straightened and run parallel to $L$, we may form paths $p$ and $q$ from paths that are
\begin{itemize}
 \item subpaths of unit cycles that $L_{\operatorname{t}(a)}$ and $L_{\operatorname{h}(a)}$ intersect, respectively; and
 \item subpaths of cycles in $\{ c_i\}_{i \in Q_0}$.
\end{itemize}
Since $p$ and $q$ are constructed from subpaths of parallel geodesic cycles, $p$ and $q$ are themselves geodesic cycles.
Thus we may define a homotopy
$$H^+: Q \times [k+ \tfrac 13, k + \tfrac 23] \to P$$
that straightens both $p$ and $q$ so that all of their arrow subpaths lie on $L_{\operatorname{t}(a)}$ and $L_{\operatorname{h}(a)}$ respectively, while leaving all other vertices of $Q$ fixed.

Finally, define a homotopy
$$H^+: Q \times [k+ \tfrac 23, k+1] \to P$$
by first translating the path $q$ along $L_{\operatorname{h}(a)}$ in the direction $\vec{\varepsilon}_k$, and the path $p$ along $L_{\operatorname{t}(a)}$ in the direction $-\vec{\varepsilon}_k$, so that the angle between $\vec{\varepsilon}_k$ and the vector $\vec{a}$ from $\operatorname{t}(a)$ to $\operatorname{h}(a)$ is less than $\tfrac{\pi}{4}$.
After this is accomplished, bring the two lines $L_{\operatorname{t}(a)}$ and $L_{\operatorname{h}(a)}$ sufficiently close together so that $\vec{a}$ is pointing in the direction $\vec{\varepsilon}_k$, to within any desired approximation.
Then the image $H^+(a,k+1)$ of $a$ is nearly parallel to $\vec{\varepsilon}_k$.

Consequently, the homotopy $H(-,t)$ rotates $a$ by $2 \pi$ as $t$ runs from $0$ to $2N$.
But this is not possible since $\Sigma$ is a compact surface, a contradiction.
\end{proof}

\begin{Proposition} \label{G}
Suppose $A$ is geodesic.
A cycle $p$ is contractible if and only if it satisfies $\overbar{p} \stackrel{\sigma}{=} 1$.
\end{Proposition}

\begin{proof}
If $p$ is contractible, then $\overbar{p} \stackrel{\sigma}{=} 1$ by Lemma \ref{Z}.i.

So assume to the contrary that there is a cycle $p$ for which $\overbar{p} \stackrel{\sigma}{=} 1$, but $p$ is not contractible, $[p] \not = 0$.
Denote by $m$ the number of times that $p$ transversely intersects the boundary $\partial P$ of the fundamental polygon $P$.

For $j \in [m]$, let $r_j \in e_{\operatorname{t}(p)}kQe_{\operatorname{t}(p)}$ be the cycle defined in (\ref{r_k}).
Denote by $k(j) \in [2N]$ the unique side of $P$ that $r_j$ transversely intersects, where the sides are ordered $1, \ldots, 2N$ counterclockwise around $P$.
For $k \in [2N]$, let $\gamma_k$ be the geodesic cycle that transversely intersects side $k$ of $P$, defined in (\ref{geodesic cycle def}).
Then $\overbar{\gamma}_{k(j)} \stackrel{\sigma}{=}\overbar{r}_j$.
Thus,
\begin{equation} \label{j in [m]}
\overbar{p} \stackrel{\sigma}{=} \prod_{j \in [m]} \overbar{r}_j \stackrel{\sigma}{=} \prod_{j \in [m]} \overbar{\gamma}_{k(j)} = \prod_{k(j) \in [2N]} \overbar{\gamma}_{k(j)}.
\end{equation}
(Note that each factor $\overbar{\gamma}_k$ may appear more than once in the product.)
Without loss of generality, we may assume that $m$ is minimal in the sense that if any proper subset of cycles $r_j$ is omitted from the product $\prod_{j \in [m]}\overbar{r}_j$, then the remaining product is not equal to $1$ (modulo $\sigma-1$). 

By assumption, $\overbar{p} \stackrel{\sigma}{=} 1$.
Therefore, by (\ref{j in [m]}),
\begin{equation} \label{pkj}
\prod_{k(j) \in [N]}\overbar{\gamma}_{k(j)} \stackrel{\sigma}{=} \prod_{k(j) \in [N+1,2N]} \overbar{\gamma}_{k(j)}^{-1} \stackrel{\sigma}{=} \prod_{k(j) \in [N+1,2N]} \overbar{\gamma}_{k(j)-N}.
\end{equation}

Fix $k \in [N]$. 
Since $A$ is geodesic, there is a subdivision containing the geodesic cycle $\gamma_{k+N}$, that is, a subdivision that `points in the opposite direction' to $\gamma_k$.
For each $k' \in [m] \setminus \{ k \}$, there is a simple matching $x \in \mathcal{S}$ obtained from the subdivision such that for some $n \geq 1$, $x^n \mid \overbar{\gamma}_k$, but $x^n \nmid \overbar{\gamma}_{k'}$. 

Furthermore, the left and right-hand sides of (\ref{pkj}) have no $\overbar{\gamma}_k$ factors in common: otherwise, if there were a $j_1,j_2 \in [m]$ such that $k(j_1) = k(j_2)+N$, then 
$$\overbar{\gamma}_{k(j_1)} \overbar{\gamma}_{k(j_2)} \stackrel{\sigma}{=} 1,$$
by Lemma \ref{Z}. 
But this would contradict the minimality of $m$.
It therefore follows that the left and right-hand sides of (\ref{pkj}) cannot be equal (modulo $\sigma-1$).
\end{proof}

The following theorem shows that a ghor algebra reflects the topology of the surface in which it is embedded.

\begin{Theorem} \label{iff}
Suppose $A$ is geodesic.
Let $p$ and $q$ be cycles in $Q$.
Then
$$[p] = [q] \ \ \text{ if and only if } \ \ \ \overbar{p} \stackrel{\sigma}{=} \overbar{q}.$$
In particular, if $\Sigma$ is smooth, then $p$ and $q$ are homologous if and only if $\overbar{p} \stackrel{\sigma}{=} \overbar{q}$.
\end{Theorem}

\begin{proof}
The forward implication was shown in Proposition \ref{how}.

So suppose $[p] \not = [q]$.
Set $i := \operatorname{t}(p^+)$ and $j := \operatorname{t}(q^+)$.
Let $s^+$, $t^+$, $q'^+$ be paths in $Q^+$ respectively from $i$ to $j$; $j$ to $i$; and $\operatorname{h}(q^+)$ to $j$.
Then $\overbar{t} \stackrel{\sigma}{=} \overbar{s}^{-1}$ and $\overbar{q}' \stackrel{\sigma}{=} \overbar{q}^{-1}$, by Lemma \ref{Z}.i.
Furthermore, since $[p] \not = [q]$, the cycle $sptq'$ is not contractible.
Thus
$$\overbar{p}\overbar{q}^{-1} \stackrel{\sigma}{=} \overbar{s} \overbar{t} \overbar{p} \overbar{q}^{-1} \stackrel{\sigma}{=} \overbar{sptq'} \stackrel{\sigma}{\not =} 1,$$
where the last inequality holds by Proposition \ref{G}.
Therefore $\overbar{p} \stackrel{\sigma}{\not =} \overbar{q}$.
\end{proof}

\begin{Remark} \label{curve} \rm{
If $\Sigma$ is smooth and $p,q$ are homologous cycles, then there is some $\ell \in \mathbb{Z}$ such that $\overbar{p} = \overbar{q} \sigma^{\ell}$, by Theorem \ref{iff}.
It may then be asked whether there is any significance to the exponent $\ell$.
We expect that the exponent is a sort of discrete measure of curvature of the surface $\Sigma$.
If this is the case, then the ghor algebra would reveal aspects of $\Sigma$ that are invisible to homologous cycles alone.
We leave an investigation of this possibility for future work.
}\end{Remark}

\begin{Proposition} \label{A}
Suppose $A$ is geodesic.
Let $p,q \in Q_{\geq 1}$ be paths for which
$$\operatorname{t}(p) = \operatorname{t}(q) \ \ \ \ \text{ and } \ \ \ \ \operatorname{h}(p) = \operatorname{h}(q).$$
Then
$$\bar{\eta}(p) = \bar{\eta}(q) \ \ \ \text{ if and only if } \ \ \ \bar{\tau}(p) = \bar{\tau}(q).$$
\end{Proposition}

\begin{proof}
(i) First suppose $\bar{\eta}(p) = \bar{\eta}(q)$.
Then
$$\bar{\tau}(p) = \bar{\eta}(p)|_{\substack{x = 1: \\ x \not \in \mathcal{S}}} = \bar{\eta}(q)|_{\substack{x = 1: \\ x \not \in \mathcal{S}}} = \bar{\tau}(q).$$

(ii) Now suppose $\bar{\tau}(p) = \bar{\tau}(q)$, and assume to the contrary that $\bar{\eta}(p) \not = \bar{\eta}(q)$.
Let $r$ be a path from $\operatorname{h}(p)$ to $\operatorname{t}(p)$, and let $u,v \in \mathbb{Z}^N$ be the classes of the cycles $rp$ and $rq$.

First suppose $u = v$.
Since $\bar{\eta}(p) \not = \bar{\eta}(q)$, there is some $x \in \mathcal{P} \setminus \mathcal{S}$ and $m \geq 1$ such that $x^m$ divides only one of $\bar{\eta}(p)$, $\bar{\eta}(q)$; say $x^m \mid \bar{\eta}(p)$ and $x^m \nmid \bar{\eta}(q)$.
Thus, since $u = v$,
$$\bar{\eta}(rp) = \bar{\eta}(rq)\sigma_{\mathcal{P}}^n$$
for some $n \geq 1$, by Proposition \ref{how}.
Therefore,
$$\bar{\tau}(r) \bar{\tau}(p) = \bar{\tau}(rp) = \bar{\eta}(rp)|_{\substack{x = 1: \\ x \not \in \mathcal{S}}} = \bar{\eta}(rq)\sigma_{\mathcal{P}}^m |_{\substack{x = 1: \\ x \not \in \mathcal{S}}} = \bar{\tau}(rq)\sigma_{\mathcal{S}}^n = \bar{\tau}(r) \bar{\tau}(q) \sigma_{\mathcal{S}}^n.$$
But this contradicts our assumption that $\bar{\tau}(p) = \bar{\tau}(q)$ since $n \geq 1$.

So suppose $u \not = v$.
Let $s \in e_{\operatorname{h}(p)}kQe_{\operatorname{h}(p)}$ be a cycle with class $v-u \not = 0$.
Then by Proposition \ref{G},
\begin{equation} \label{notyet}
\bar{\tau}(s) \stackrel{\sigma}{\not =} 1.
\end{equation}

Since the cycles $rsp$ and $rq$ both have class $v$, we have
\begin{equation} \label{a2}
\bar{\eta}(rsp) \stackrel{\sigma}{=} \bar{\eta}(rq),
\end{equation}
by Proposition \ref{how}.
Whence
$$\bar{\tau}(s)\bar{\tau}(rq) = \bar{\tau}(s) \bar{\tau}(rp) = \bar{\tau}(rsp) = \bar{\eta}(rsp)|_{\substack{x' = 1: \\ x' \not \in \mathcal{S}}} \stackrel{\sigma}{=} \bar{\eta}(rq)|_{\substack{x' = 1: \\ x' \not \in \mathcal{S}}} = \bar{\tau}(rq).$$
Consequently, $\bar{\tau}(s) \stackrel{\sigma}{=} 1$ since $k[\mathcal{S}]$ is an integral domain, contrary to (\ref{notyet}).
\end{proof}

\begin{Corollary} \label{suffices}
If $A$ is geodesic, then
$$A := kQ/\operatorname{ker}\eta \cong kQ/\operatorname{ker}\tau.$$
In particular, it suffices to only consider the simple matchings of $Q$ to determine the relations of $A$.
\end{Corollary}

\begin{proof}
Follows from Proposition \ref{A}.
\end{proof}

\begin{Corollary} \label{injective}
If $A$ is geodesic, then the algebra homomorphism $\tau: kQ \to M_{|Q_0|}(k[\mathcal{S}])$ induces an injective algebra homomorphism on $A$,
$$\tau: A \hookrightarrow M_{|Q_0|}(k[\mathcal{S}]).$$
Consequently, the ring $R$ given in (\ref{R def}) is isomorphic to the center of $A$.
\end{Corollary}

\begin{proof}
By the definition of $A$, two paths $p,q \in e_jAe_i$ satisfy $\bar{\eta}(p) = \bar{\eta}(q)$ if and only if $p = q$.
Furthermore, if $A$ is geodesic, then $\bar{\tau}(p) = \bar{\tau}(q)$ if and only if $p = q$, by Corollary \ref{suffices}.
\end{proof}

\section{Central geometry and endomorphism ring structure}

Let $A = kQ/ \ker \eta$ be a ghor algebra, and let $R$ and $S$ be as defined in (\ref{R def}) and (\ref{S def}).
Recall that $R$ is the center of $A$ whenever $A$ is geodesic, by Corollary \ref{injective}. 
Throughout this section, set $\overbar{p} := \bar{\tau}(p)$ for $p \in e_i kQ e_j$, and $\sigma := \sigma_{\mathcal{S}} = \prod_{x \in \mathcal{S}} x$.

\subsection{Nonnoetherian central geometry} \label{ncg}

In this section we will show that the centers of geodesic ghor algebras are typically nonnoetherian, and we will use the notion of depiction (Definition \ref{depiction definition}) to make sense of their geometry.

\begin{Lemma} \label{finitely generated}
The cycle algebra $S$ is a finitely generated $k$-algebra.
\end{Lemma}

\begin{proof}
The dimer quiver $Q$ is finite, and so there are only a finite number of cycles in $Q$ that do not have cyclic proper subpaths.
\end{proof}

\begin{Lemma} \label{an}
Let $A$ be a ghor algebra.
Then $R$ is noetherian if and only if for each monomial $g \in S$, there is an $n \geq 1$ such that $g^n$ is in $R$.
\end{Lemma}

\begin{proof} \

$\Rightarrow$: First suppose there is a monomial $g \in S$ such that $g^n \not \in R$ for each $n \geq 1$.
By Lemma \ref{Z}.i, there is an $m \geq 1$ such that for each $n \geq 1$, $g^n \sigma^m$ is in $R$.
Consider the infinite ascending chain of ideals
$$0 \subseteq \sigma^mR \subseteq (1,g) \sigma^m R \subseteq (1,g,g^2)\sigma^m R \subseteq (1,g,g^2,g^3) \sigma^m R \subseteq \cdots.$$

We claim that each inclusion is strict.
Assume to the contrary that there is some $n \geq 1$ and $r_0, \ldots, r_{n-1} \in R$ for which
$$g^n \sigma^m = \sum_{i = 0}^{n-1} r_i g^i \sigma^m.$$
Then, since $k[\mathcal{S}]$ is an integral domain,
$$g^n - \sum_{i = 1}^{n-1} r_i g^i = r_0 \in R.$$
Furthermore, $R$ is generated by monomials in the polynomial ring $k[\mathcal{S}]$, and so each monomial summand of the polynomial $r_0$ is in $R$.
Thus, since $g^n \not \in R$, there is some $1 \leq i \leq n-1$ such that $g^{n-i}$ is a monomial summand of $r_i$.
But then $r_i$ is not in $R$, a contradiction.
Therefore $R$ is nonnoetherian.

$\Leftarrow$: Now suppose that for each monomial $g \in S$, there is an $n \geq 1$ such that $g^n$ is in $R$.
By Lemma \ref{finitely generated}, $S$ is generated by a finite number of cycles $s_j$,
\begin{equation} \label{Sfin}
S = k[\overbar{s}_j \, | \,j \in [1, \ell]].
\end{equation}
By assumption, for each $j \in [0, \ell]$, there is a minimum $n_j \geq 1$ such that $\overbar{s}_j^{n_j}$ is in $R$.
Consider the two sets of cycles
\begin{align*}
X & := \left\{ \text{cycles in $R$ that contain at most one of each vertex} \right\}, \\
Y & := \left\{ \prod_{j = 1}^{\ell} \overbar{s}_j^{m_j} \in S \setminus R \ | \ m_j \in [0, n_j-1] \right\}.
\end{align*}

We claim that
$$R = k[X, XY \cap R].$$
Indeed, let $r \in X$, $s_1, s_2 \in Y$, and suppose $(rs_1)s_2 \in (XY \cap R)Y \cap R$.
Then
$$(rs_1)s_2 = r(s_1s_2) \in \left\{ \begin{array}{ll} XY \cap R & \text{ if } s_1s_2 \not \in R\\ k[X] & \text{ if } s_1s_2 \in R \end{array} \right.$$
Consequently,
$$(XY \cap R)Y \cap R \subseteq k[X, XY \cap R],$$
proving our claim.
But $|Y|< \infty$, and $|X| < \infty$ since $Q$ has a finite number of vertices.
Whence $|XY \cap R| \leq |X| |Y| < \infty$.
Thus $R = k[X, XY \cap R]$ is a finitely generated $k$-algebra, and therefore noetherian.
\end{proof}

If $\Sigma$ is a torus, then a ghor algebra $A$ is geodesic if and only if it is noetherian, if and only if its center $R$ is noetherian, if and only if $A$ is a finitely generated $R$-module \cite[Theorem 1.1]{B2}.
If $\Sigma$ is hyperbolic, then only one direction of the implication survives:

\begin{Lemma}
If $R$ is noetherian, then
\begin{enumerate}
 \item $A$ is a finitely generated $R$-module;
 \item $A$ is noetherian; and
 \item $A$ is geodesic.
\end{enumerate}
\end{Lemma}

\begin{proof}
Suppose $R$ is noetherian.

(1): Recall the finite generating set (\ref{Sfin}) of $S$.
By Lemma \ref{an}, there are minimum integers $n_i \geq 1$ satisfying $\overbar{s}_i^{n_i} \in R$.
Set
$$m := \operatorname{max}\{n_1, \ldots, n_{\ell} \}.$$
Let $L$ be the length of the longest path in $Q$ with no cyclic proper subpath.
Then each path of length $\geq (L+1)\ell m$ will contain a cyclic subpath whose $\bar{\tau}$-image is in $R$.
Therefore $A$ is generated as an $R$-module by the set of all paths in $Q$ of length at most $(L+1)\ell m$.

(2): Follows from (1) and the assumption that $R$ is noetherian. 

(3): Suppose $A$ is not geodesic.
Then there is a vertex $i \in Q_0$ and direction $k \in [2N]$ for which every cycle at $i$ parallel to $k$ is not geodesic.
Let $p$ be a cycle at $i$ parallel to $k$, whose lift $p^+$ to the cover $Q^+$ contains no cyclic subpaths modulo $\ker \eta$ (though some cyclic permutation of $p^+$ necessarily contains a cyclic subpath since $p$ is not geodesic).
Let $q^+$ be formed from a cyclic permutation of $p^+$ by removing at least one cyclic subpath.
Then $\overbar{q}^n$ is in $S \setminus R$ for all $n \geq 1$, by Theorem \ref{iff}. 
Therefore $R$ is nonnoetherian by Lemma \ref{an}.
\end{proof}

In the following we use the assumption that $k$ is uncountable.

\begin{Lemma} \label{surjective}
The morphisms
\begin{equation*} \label{max surjective}
\begin{array}{rcl}
\kappa_{S/\hat{Z}}: \operatorname{Max}S \to \operatorname{Max}R, & \ & \mathfrak{n} \mapsto \mathfrak{n} \cap R,\\
\iota_{S/R}: \operatorname{Spec}S \to \operatorname{Spec}R, & & \mathfrak{q} \mapsto \mathfrak{q} \cap R,
\end{array}
\end{equation*}
are well-defined and surjective.
\end{Lemma}

\begin{proof}
(i) We first claim that the map $\kappa_{S/R}$ is well-defined.
Indeed, let $\mathfrak{n}$ be in $\operatorname{Max}S$.
By Lemma \ref{finitely generated}, $S$ is of finite type, and by assumption $k$ is algebraically closed.
Therefore the intersection $\mathfrak{n} \cap R$ is a maximal ideal of $R$ (e.g., \cite[Lemma 2.1]{B3}).

(ii) We claim that $\kappa_{S/R}$ is surjective.
Fix $\mathfrak{m} \in \operatorname{max}R$.
Then $S\mathfrak{m}$ is a proper ideal of $S$ since $S$ is a subalgebra of the polynomial ring $k[\mathcal{S}]$.
Thus, since $S$ is noetherian, there is a maximal ideal $\mathfrak{n} \in \operatorname{Max}S$ containing $S\mathfrak{m}$.
Whence,
$$\mathfrak{m} \subseteq S\mathfrak{m} \cap R \subseteq \mathfrak{n} \cap R.$$
But $\mathfrak{n} \cap R$ is a maximal ideal of $R$ by Claim (i).
Therefore $\mathfrak{m} = \mathfrak{n} \cap R$.

(iii) It is clear that the map $\iota_{S/R}$ is well-defined.
Finally, we claim that $\iota_{S/R}$ is surjective.
By \cite[Lemma 3.6]{B3}, if $D$ is a finitely generated algebra over an uncountable field $k$, and $C \subseteq D$ is a subalgebra, then $\iota_{D/C}: \operatorname{Spec}D \to \operatorname{Spec}C$ is surjective if and only if $\kappa_{D/C}:\operatorname{Max}D \to \operatorname{Max}C$ is surjective.
Therefore, $\iota_{S/R}$ is surjective by Claim (ii).
\end{proof}

\begin{Lemma} \label{each vertex}
Let $\mathfrak{n} \in \operatorname{Max}S$ be a maximal ideal.
Suppose that each cycle $p + \ker \eta$ that passes through each vertex of $Q$ satisfies $\overbar{p} \in \mathfrak{n}$.
Then the localization $R_{\mathfrak{n} \cap R}$ is nonnoetherian.
\end{Lemma}

\begin{proof}
Let $p$ be a cycle for which $\overbar{p}$ is in $R \setminus \mathfrak{n}$.
Then $\overbar{p}^m \not \in \mathfrak{n}$ for each $m \geq 1$.
Thus, by assumption, $p^m + \ker \eta$ does not pass through each vertex of $Q$.
In particular, there is a vertex $j \in Q_0$ such that $e_j$ is not a subpath of $p^m + \ker \eta$ for $m \geq 1$.

Since $\overbar{p}$ is in $R$, there is a cycle $q \in e_j kQ e_j$ such that $\overbar{q} = \overbar{p}$.
By Theorem \ref{iff}, we have $[p] \not = [q]$.
It thus suffices to suppose that $q$ has a cyclic proper subpath $s$ that runs along one of the edges of the fundamental polygon $P$; see Figure \ref{ghm}.i.

Let $\ell \geq 1$ be such that $\sigma^{\ell}_{\operatorname{t}(s)} + \ker \eta$ passes through each vertex of $Q$.
Then $\overbar{s}^n\sigma^{\ell}$ is in $R$ for each $n \geq 1$.
However, there is no cycle that intersects $p$ with monomial $\overbar{s}^n$ for any $n \geq 1$, since such a cycle would necessarily be in the same class as $s^n$, again by Theorem \ref{iff}, and $p + \ker \eta$ and $q +\ker \eta$ do not intersect.
Thus $\overbar{s}^n$ is not in $R_{\mathfrak{n} \cap R}$ for each $n \geq 1$.
But then the ascending chain of ideals of $R_{\mathfrak{n} \cap R}$,
$$\overbar{s}\sigma^{\ell} \subset (\overbar{s}, \overbar{s}^2)\sigma^{\ell} \subset (\overbar{s}, \overbar{s}^2, \overbar{s}^3)\sigma^{\ell} \subset \cdots,$$
does not stabilize.
Therefore $R_{\mathfrak{n} \cap R}$ is nonnoetherian.
\end{proof}

\begin{figure}
$$\begin{array}{ccccc}
\includegraphics[width=4cm]{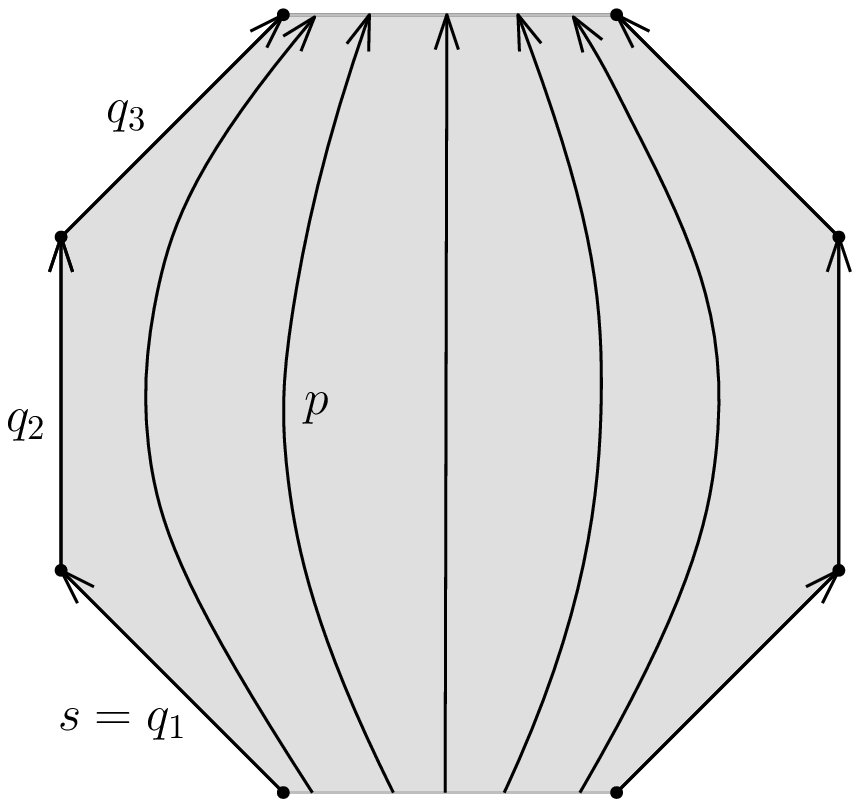} & \ \ \ \ &
\includegraphics[width=4cm]{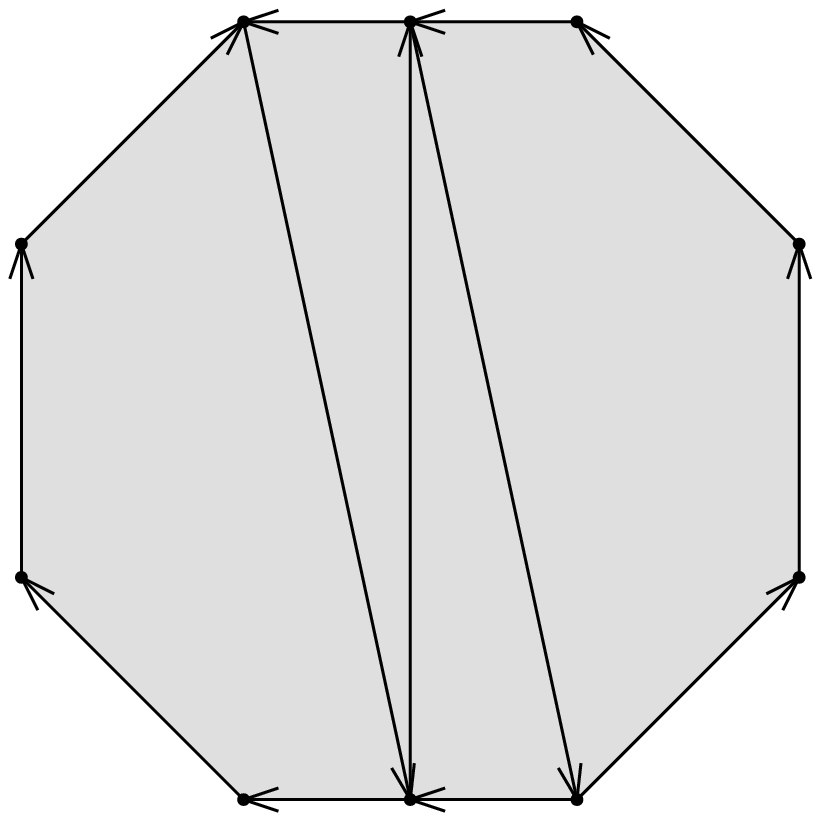} & \ \ \ \ &
\includegraphics[width=4cm]{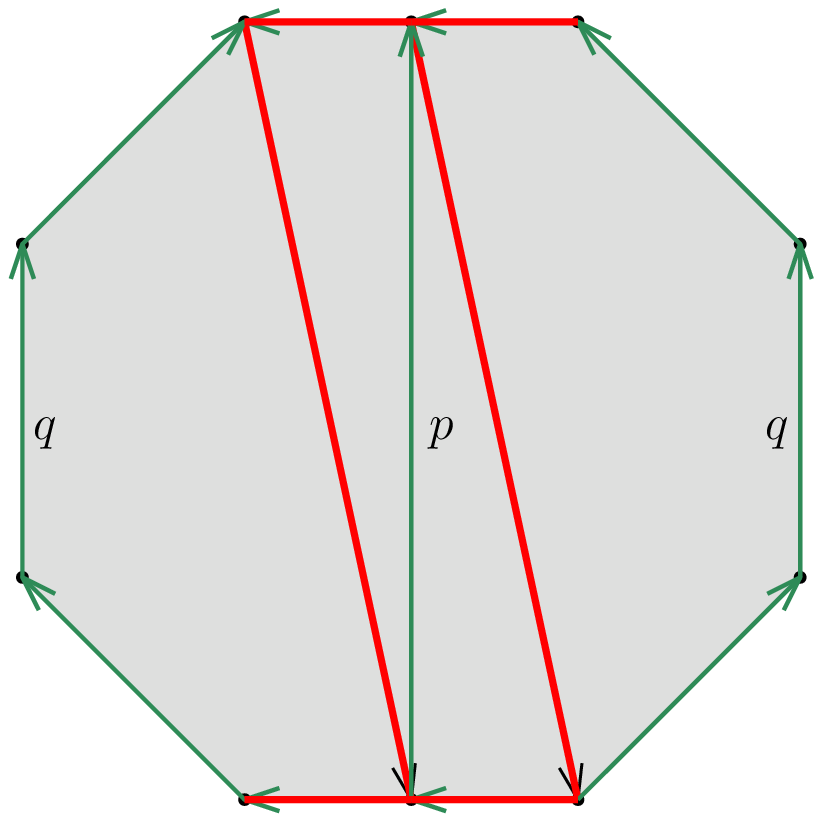}\\
(i) & & (ii) & & (iii)
\end{array}$$
\caption{(i): Setup for Lemma \ref{each vertex} (with irrelevant paths omitted).
(ii): A specific example of (i), described in Example \ref{generalized polynomial}.
(iii): A subdivision consisting of two columns (drawn in green).}
\label{ghm}
\end{figure}

\begin{Example} \label{generalized polynomial} \rm{
In Lemma \ref{each vertex} we considered a ghor algebra that contains a cycle $p$ whose monomial $\overbar{p}$ is in $R$, yet $p^m + \ker \eta$ does not pass through each vertex of $Q$ for all $m \geq 1$.
A priori it is unclear whether such ghor algebras exist; however, an example with this property is given in Figure \ref{ghm}.ii.
This particular example may be regarded as a two-vertex generalization of the polynomial ghor algebra given in Example \ref{fig:polynomial}.
}\end{Example}

\begin{Theorem} \label{depiction theorem}
If the center $R$ of a geodesic ghor algebra $A$ is nonnoetherian, then the cycle algebra $S$ of $A$ is a depiction of $R$.
\end{Theorem}

\begin{proof}
Suppose that $R$ is nonnoetherian.

(i) We first claim that the locus $U_{S/R}$ is nonempty.

Choose a maximal ideal $\mathfrak{n}$ of $S$ for which $\sigma \not \in \mathfrak{n}$ (that is, for each simple matching $x \in \mathcal{S}$, there is a nonzero constant $c \in k^*$  such that $(x - c)k[\mathcal{S}] \cap S \subset \mathfrak{n}$).
Let $m \geq 1$ be such that $\sigma_i^m + \ker \eta$ passes through each vertex of $Q$.
Let $s$ be any cycle in $Q$.
Then $\overbar{s} \sigma^m = \overbar{s\sigma_{\operatorname{t}(p)}^m} \in R$.
Whence
$$\overbar{s} = \frac{\overbar{s}\sigma^m}{\sigma^m} \in R_{\mathfrak{n} \cap R}.$$
Thus $S \subset R_{\mathfrak{n} \cap R}$, and therefore $S_{\mathfrak{n}} \subseteq R_{\mathfrak{n} \cap R} \subseteq S_{\mathfrak{n}}$.
Consequently, $\mathfrak{n} \in U_{S/R}$.

(ii) Let $\mathfrak{n} \in \operatorname{Max}S$, and suppose the localization $R_{\mathfrak{n} \cap R}$ is noetherian.
We claim that $R_{\mathfrak{n} \cap R} = S_{\mathfrak{n}}$.

Fix a cycle $s$.
Since $R_{\mathfrak{n} \cap R}$ is noetherian, there is a cycle $p$ which passes through each vertex of $Q$ such that $\overbar{p} \not \in \mathfrak{n}$, by Lemma \ref{each vertex}.
Since $p$ passes through each vertex of $Q$, the monomials $\overbar{p}$ and $\overbar{s} \overbar{p}$ are both in $R$.
Thus
$$\overbar{s} = \frac{\overbar{s} \overbar{p}}{\overbar{p}} \in R_{\mathfrak{n} \cap R}.$$
Whence $S \subset R_{\mathfrak{n} \cap R}$.
Therefore $R_{\mathfrak{n} \cap R} = S_{\mathfrak{n}}$.

(iii) Finally, the morphism $\operatorname{Spec}S \to \operatorname{Spec}R$ is surjective by Lemma \ref{surjective}.
\end{proof}

\subsection{The Krull dimension of the center} \label{Krull section}

Suppose $A$ is geodesic.
Denote by $T$ the subalgebra of $S$ generated by $\sigma$ and the $\bar{\tau}$-images of the $2N$ geodesic cycles $\gamma_1, \ldots, \gamma_{2N}$ in Definition \ref{geocom}; without loss of generality we may assume
\begin{equation} \label{complementary sigma}
\gamma_k \gamma_{k+N} \stackrel{\sigma}{=} 1,
\end{equation}
by Lemma \ref{Z}.i.

\begin{Lemma} \label{birational}
The following inclusions hold:
\begin{itemize}
 \item[(i)] $S \subset T[\sigma^{-1}]$.
 \item[(ii)] $S \subset R[\sigma^{-1}]$.
 \item[(iii)] $R[\sigma^{-1}] = T[\sigma^{-1}] = S[\sigma^{-1}]$.
\end{itemize}
In particular, the algebraic varieties $\operatorname{Max}R$, $\operatorname{Max}T$, and $\operatorname{Max}S$ are birationally equivalent.
\end{Lemma}

\begin{proof}
(i) We first claim that $S \subset T[\sigma^{-1}]$.
Let $s$ be a cycle.
Let $p^+$ be a path lying in the fundamental polygon $P$ from $\operatorname{t}(s^+)$ to a vertex $i_0^+ \in \pi^{-1}(i)$, and let $q^+$ be a path lying in $P$ from a vertex $i_1^+ \in \pi^{-1}(i)$ to $\operatorname{t}(s^+)$.
Then there is an $\ell \geq 0$ such that $\overbar{pq} = \sigma^{\ell}$, by Lemma \ref{Z}.i.

Construct a path $t^+$ in $Q^+$ by concatenating lifts of cycles in $T$ from $i_0^+$ to $i_1^+$.
Then
$$\operatorname{t}(t^+) = \operatorname{t}((psq)^+) \ \ \ \text{ and } \ \ \ \operatorname{h}(t^+) = \operatorname{h}((psq)^+).$$
Thus there is an $m \in \mathbb{Z}$ such that $\bar{t} = \overbar{psq}\sigma^m$, by Lemma \ref{Z}.ii.
But then
$$\overbar{s} = \sigma^{-\ell}\overbar{psq} = \bar{t}\sigma^{-\ell -m} \in T[\sigma^{-1}].$$

(ii) We now claim that $S \subset R[\sigma^{-1}]$.
Let $s$ be a cycle.
Let $p^+$ be a cycle in $Q^+$ such that $p = \pi(p^+)$ passes through each vertex of $Q$.
By Lemma \ref{Z}.i, there is an $\ell \geq 1$ such that $\bar{p} = \sigma^{\ell}$.
Therefore
$$\bar{s} = \sigma^{-\ell} \overbar{ps} \in R[\sigma^{-1}].$$

(iii) We have
$$R[\sigma^{-1}] = S[\sigma^{-1}] \stackrel{\textsc{(i)}}{\subseteq} T[\sigma^{-1}] \subseteq S[\sigma^{-1}] \stackrel{\textsc{(ii)}}{\subseteq} R[\sigma^{-1}],$$
where (\textsc{i}) holds by Claim (i), and (\textsc{ii}) holds by Claim (ii).
\end{proof}

\begin{Lemma} \label{S'1}
The Krull dimension of $T$ is $N+1$.
\end{Lemma}

\begin{proof}
For $j \in [N]$, set
$$\mathfrak{p}_j := (\sigma, \overbar{\gamma}_1, \ldots, \overbar{\gamma}_N, \overbar{\gamma}_{N+1}, \ldots, \overbar{\gamma}_{j+N})T,$$
and consider the chain of ideals of $T$,
\begin{equation} \label{chain}
0 \subseteq \mathfrak{p}_0 \subseteq \mathfrak{p}_1 \subseteq \cdots \subseteq \mathfrak{p}_N.
\end{equation}

(i) We first claim that each $\mathfrak{p}_j$ is prime.

Indeed, let $p,q$ be cycles for which $\overbar{p}, \overbar{q} \in T$, and suppose $\overbar{p} \overbar{q} \in \mathfrak{p}_j$.
Then there is a $k \in [j+N]$ and a monomial $\overbar{t} \in T$ such that
$$\overbar{p} \overbar{q} = \overbar{t} \overbar{\gamma}_k.$$
By Theorem \ref{iff}, the cycles $pq$ and $t\gamma_k$ are in the same class.
But then $\gamma_k$ is a factor of $p$ or a factor $q$, whence $p \in \mathfrak{p}_j$ or $q \in \mathfrak{p}_j$.

(ii) We now claim that the inclusions in (\ref{chain}) are strict.

Fix $j \in [N-1]$ and $k \in [j+1, N]$.
Assume to the contrary that $\overbar{\gamma}_{k+N} \in \mathfrak{p}_j$.
By assumption, $\sigma \nmid \overbar{\gamma}_{k+N}$.
Whence $\overbar{\gamma}_{k+N} \not \in \sigma T$.
Thus, since $T$ is toric, there is some $i \in [j+N]$ and a monomial $\overbar{t} \in T$ such that $\overbar{\gamma}_{k+N} = \overbar{t} \overbar{\gamma}_i$.
But then by Theorem \ref{iff}, $\gamma_{k+N}$ and $t \gamma_i$ are in the same class, a contradiction.
Therefore the chain (\ref{chain}) is strict.

(iii) Finally, the chain (\ref{chain}) is maximal: Let $\mathfrak{p}$ be a prime ideal of $T$ generated by monomials, and let $k \in [N]$.
Then $\sigma$ is in $\mathfrak{p}$ if and only if $\overbar{\gamma}_k$ or $\overbar{\gamma}_{k+N}$ is in $\mathfrak{p}$ by (\ref{complementary sigma}).
\end{proof}

\begin{Proposition} \label{dimensions}
The Krull dimensions of $R$, $S$, and $T$ are equal.
\end{Proposition}

\begin{proof}
If $B$ is a nonnoetherian integral domain depicted by $C$, then $\operatorname{dim}B = \operatorname{dim}C$, by \cite[]{B3}.
Furthermore, if an integral domain is finitely generated over $k$, then its Krull dimension and transcendence degree are equal. 
Therefore
\begin{multline*}
\operatorname{dim}R \stackrel{\textsc{(i)}}{=} \operatorname{dim}S \stackrel{\textsc{(ii)}}{=} \operatorname{trdeg}_k (\operatorname{Frac}S)
= \operatorname{trdeg}_k (\operatorname{Frac}S[\sigma^{-1}])\\
\stackrel{\textsc{(iii)}}{=} \operatorname{trdeg}_k (\operatorname{Frac}T[\sigma^{-1}])
= \operatorname{trdeg}_k (\operatorname{Frac}T)
 \stackrel{\textsc{(iv)}}{=} \operatorname{dim}T,
\end{multline*}
where (\textsc{i}) holds since $S$ is a depiction of $R$ by Theorem \ref{depiction theorem}; (\textsc{iii}) holds by Lemma \ref{birational}; and (\textsc{ii}) and (\textsc{iv}) hold since $S[\sigma^{-1}]$ and $T[\sigma^{-1}]$ are finitely generated over $k$ by Lemma \ref{finitely generated}.
\end{proof}

\begin{Theorem} \label{KZA}
The Krull dimensions of $R$ and $S$ satisfy
\begin{equation} \label{dimR}
\operatorname{dim}R = \operatorname{dim}S = N + 1.
\end{equation}
In particular, if $\Sigma$ is a smooth genus $g \geq 1$ surface, then
$$\operatorname{dim}R = \operatorname{rank}H_1(\Sigma) + 1 = 2g + 1.$$
\end{Theorem}

\begin{proof}
We have $\operatorname{dim}R = \operatorname{dim}S = \operatorname{dim}T$ by Proposition \ref{dimensions}.
The equalities (\ref{dimR}) therefore hold by Lemma \ref{S'1}.
\end{proof}

\subsection{Endomorphism ring structure} \label{endo section}

Let $A$ be a ghor algebra with center $R$ and cycle algebra $S$.
In the following, set $\overbar{p} := \bar{\tau}(p)$ for $p \in e_jkQe_i$.

\begin{Lemma} \label{endo lemma}
Suppose $A$ is geodesic.
Fix $i,j \in Q_0$, and let $f \in \operatorname{End}_R(Ae_i)$.
Then there is some $g \in k[\mathcal{S}]$ such that $\overbar{f(p)} = g\overbar{p}$ for all $p \in e_jAe_i$.
\end{Lemma}

\begin{proof}
Fix $i,j \in Q_0$, and let $f \in \operatorname{End}_R(Ae_i)$.
Consider a nontrivial path $p \in e_jkQe_i$.
Since $A$ is geodesic, there is a simple matching $x \in \mathcal{S}$ such that $x \mid \overbar{p}$, by Proposition \ref{arrow simple}.
Let $m \geq 1$ be such that $x^m \mid \overbar{p}$ and $x^{m+1} \nmid \overbar{p}$.
Since $x$ is simple, there is a path $q \in e_jkQe_i$ for which $x \nmid \overbar{q}$.

Let $r^+$ be a path in $Q^+$ from $\operatorname{h}(p^+)$ to $\operatorname{t}(p^+)$.
Let $\ell \geq 0$ be sufficiently large so that $\overbar{pr}\sigma^{\ell}$ and $\overbar{qr}\sigma^{\ell}$ are both in $R$.
Then, since $f$ is an $R$-module homomorphism, we have
$$(p \sigma_i^{\ell} r)f(q) = f(p\sigma_i^{\ell} rq) = f(q\sigma_i^{\ell}rp) = (q\sigma_i^{\ell}r)f(p).$$
Thus, since $k[\mathcal{S}]$ is an integral domain,
\begin{equation} \label{fz}
\frac{\overbar{f(q)}}{\overbar{q}} = \frac{\overbar{f(p)}}{\overbar{p}}.
\end{equation}
But $x^m \mid \overbar{p}$ and $x \nmid \overbar{q}$, and so (\ref{fz}) implies $x^m \mid \overbar{f(p)}$.
Therefore, since $x \in \mathcal{S}$ was an arbitrary simple matching for which $x \mid \overbar{p}$, we have
$$\overbar{p} \mid \overbar{f(p)}.$$

Set
\begin{equation} \label{g :=}
g := \frac{\overbar{f(p)}}{\overbar{p}} \in k[\mathcal{S}].
\end{equation}
Then $\overbar{f(p)} = g\overbar{p}$ for all $p \in e_jAe_i$, again by (\ref{fz}).
\end{proof}

\begin{Proposition}
Suppose $A$ is geodesic.
Then the topological ring
$$A_{\sigma} := A \otimes_R R[\sigma^{-1}]$$
is an endomorphism ring over its center: for each $i \in Q_0$, we have
$$A_{\sigma} \cong \operatorname{End}_{R_{\sigma}}(A_{\sigma} e_i).$$
\end{Proposition}

\begin{proof}
Fix vertices $i,j \in Q_0$, and a left $R$-module endomorphism $f \in \operatorname{End}_R(Ae_i)$.
By the linearity of $f$, we may assume that the polynomial $g \in k[\mathcal{S}]$ corresponding to $f$ and $j$, as defined in Lemma \ref{endo lemma}, is a monomial.
Let $p \in e_jkQe_i$ be a path; then $q := f(p)$ is also a path since $\overbar{q} = g\overbar{p}$ is a monomial.
Let $t^+$ be a path in the cover $Q^+$ from $\operatorname{h}(p^+)$ to $\operatorname{h}(q^+)$.
Thus, by Lemma \ref{Z}.ii, there is an $\ell \in \mathbb{Z}$ such that
$$\overbar{tp} = \overbar{q} \sigma^{\ell}.$$
But then the path
$$t \sigma^{-\ell} \in e_kA_{\sigma}e_j$$
satisfies $\overbar{t}\sigma^{-\ell} = g$.
Therefore $f$ acts on $e_jA_{\sigma}e_i$ by left multiplication by $t \sigma^{-\ell}$.

Conversely, every element of $A_{\sigma}$ defines a left $R_{\sigma}$-module endomorphism of $A_{\sigma}e_i$ by left multiplication.
\end{proof}

\begin{Lemma} \label{al}
Suppose $A$ is geodesic.
If $p, q \in kQe_j$ are paths satisfying $\overbar{p} = \overbar{q}$, then $\operatorname{h}(p^+) = \operatorname{h}(q^+)$.
\end{Lemma}

\begin{proof}
Suppose $p,q \in kQe_j$ satisfy $\overbar{p} = \overbar{q}$.
Let $r^+$ be a path in the cover $Q^+$ from $\operatorname{h}(p^+)$ to $\operatorname{h}(q^+)$.
Then $\overbar{rp} = \overbar{q} \sigma^{\ell}$ for some $\ell \in \mathbb{Z}$, by Lemma \ref{Z}.ii.
Whence $\overbar{r} = \sigma^{\ell}$ since $\overbar{p} = \overbar{q}$.
But then $\operatorname{t}(r^+) = \operatorname{h}(r^+)$ by Proposition \ref{G}.
\end{proof}

Recall the noetherian locus $U_{S/R}$ defined in (\ref{U}), which is an open dense subset of the algebraic variety $\operatorname{Max}S$ of the cycle algebra $S$.

\begin{Theorem} \label{endo}
Suppose $A$ is geodesic.
At each point $\mathfrak{m} \in \operatorname{Max}R$ which lifts to the noetherian locus $U_{S/R} \subseteq \operatorname{Max}S$, that is, for which the localization $R_{\mathfrak{m}}$ is noetherian, the localization $A_{\mathfrak{m}} := A \otimes_R R_{\mathfrak{m}}$ is an endomorphism ring over its center: for each $i \in Q_0$, we have
$$A_{\mathfrak{m}} \cong \operatorname{End}_{R_{\mathfrak{m}}}(A_{\mathfrak{m}}e_i).$$
\end{Theorem}

\begin{proof}
Fix $\mathfrak{m} \in \operatorname{Max}R$ for which $R_{\mathfrak{m}}$ is noetherian; vertices $i,j \in Q_0$; and a left $R$-module endomorphism $f \in \operatorname{End}_R(Ae_i)$.
By Theorem \ref{depiction theorem}, $S$ is a depiction of $R$, and thus there is an $\mathfrak{n} \in \operatorname{Max}S$ such that $\mathfrak{n} \cap R = \mathfrak{m}$ and $S_{\mathfrak{n}} = R_{\mathfrak{m}}$.

By the linearity of $f$, we may assume that the polynomial $g \in k[\mathcal{S}]$ corresponding to $f$ and $j$, as defined in Lemma \ref{endo lemma}, is a monomial.
Thus, by Lemma \ref{al} it suffices to show that there is a path $t \in kQe_j$ and cycle $s$ with $\overbar{s} \not \in \mathfrak{n}$, for which
$$g = \overbar{t} \overbar{s}^{-1} \in \bar{\tau}(e_{\operatorname{h}(t)}Ae_j) \otimes_R R_{\mathfrak{m}}.$$

Let $m \geq 0$ be maximum such that $\sigma^m \mid g$ and $\sigma^{m+1} \nmid g$.
If $g = \sigma^m$, then we may take $t = \sigma_{j}^m$ by Proposition \ref{G}.
So suppose there is a simple matching $x \in \mathcal{S}$ such that $x \nmid g \sigma^{-m}$.
Since $x$ is simple, there is a path $p \in e_jkQe_i$ such that $x \nmid \overbar{p}$.
Set $q := f(p)$ and $k := \operatorname{h}(q)$; then $\overbar{q} = \overbar{f(p)} = g \overbar{p}$.

Again since $x$ is simple, there are paths $r \in e_ikQe_k$ and $t_1 \in e_k kQe_j$ for which $x \nmid \overbar{r}$ and $x \nmid \overbar{t}_1$.
Then $t_1pr \in e_k kQ e_k$ is a cycle satisfying
\begin{equation} \label{t1pr}
x \nmid \overbar{t_1pr}.
\end{equation}
Furthermore,
$$x \nmid g\overbar{p}\overbar{r} \sigma^{-m} = \overbar{qr} \sigma^{-m}.$$
Thus, by Theorem \ref{iff} and the fact that $\overbar{pr} \mid \overbar{qr}$, we may choose $t_1$ so that each component of the class $[t_1pr] \in \mathbb{Z}^N$ satisfies
\begin{equation} \label{qk}
\operatorname{sign}([t_1pr]_{\nu}) = \operatorname{sign}( [qr]_{\nu}) \ \ \ \ \text{ and } \ \ \ \ |[t_1pr]_{\nu}| \leq |[qr]_{\nu}|,
\end{equation}
with $\nu \in [N]$.

Choose a cycle $t_3$ (at any vertex of $Q$) such that
\begin{itemize}
 \item[(i)] $[t_3] = [qr] - [t_1pr]$, and
 \item[(ii)] $\overbar{t}_3$ is minimally divisible by $\sigma$ such that (i) holds.
\end{itemize}
Set $\ell := \operatorname{t}(t_3)$.

Since $R_{\mathfrak{n}\cap R}$ is noetherian, there is a cycle $t_4 t_2 = t_4e_{\ell}t_2 \in e_k kQ e_k$ that passes through $\ell$ for which $\overbar{t_4t_2} \not \in \mathfrak{n}$, by Lemma \ref{each vertex}.
By (i) and Theorem \ref{iff} we have
$$\overbar{qr} = \overbar{t_3} \cdot \overbar{t_1pr} \sigma^n$$
for some $n \in \mathbb{Z}$.
Furthermore, by (ii), (\ref{t1pr}), and (\ref{qk}), we have $n \geq 0$.
Therefore, setting $t := t_4t_3t_2t_1$, we find
$$g = \frac{\overbar{q}}{\overbar{p}} = \frac{\overbar{t_4t_3t_2t_1p}}{\overbar{t_4t_2} \cdot \overbar{p}} = \frac{\overbar{t}}{\overbar{t_4t_2}} \in \bar{\tau}(e_kAe_j) \otimes_R R_{\mathfrak{m}}.$$
Consequently, $g$ is the $\bar{\tau}$-image of the cycle
$$\frac{t}{\overbar{t_4t_2}} \in e_k A_{\mathfrak{m}} e_j.$$

Conversely, every element of $A_{\mathfrak{m}}$ defines a left $R_{\mathfrak{m}}$-module endomorphism of $A_{\mathfrak{m}}e_i$ by left multiplication.
\end{proof}

\begin{Remark} \rm{
Theorem \ref{endo} is a generalization of the well known fact that cancellative dimer algebras $A$ on a torus are endomorphism rings: for each $i \in Q_0$, there is an algebra isomorphism $A \cong \operatorname{End}_R(Ae_i)$.
Indeed, a cancellative dimer algebra on a torus is a geodesic ghor algebra for which $U_{S/R} = \operatorname{Max}S = \operatorname{Max}R$ (noting that, in this case, $R = S$).
Thus Theorem \ref{endo} implies that for each $\mathfrak{m} \in \operatorname{Max}R$, the localization $A_{\mathfrak{m}}$ is an endomorphism ring, $A_{\mathfrak{m}} \cong \operatorname{End}_{R_{\mathfrak{m}}}(A_{\mathfrak{m}}e_i)$.
This in turn implies the isomorphism $A \cong \operatorname{End}_R(Ae_i)$.
}\end{Remark}

\ \\
\textbf{Acknowledgments.}
The authors were supported by the Austrian Science Fund (FWF) grant P 30549-N26.
The first author was also supported by the Austrian Science Fund (FWF) grant W 1230 and by a Royal Society Wolfson Fellowship.

\bibliographystyle{hep}
\def\cprime{$'$} \def\cprime{$'$}

\end{document}